\documentclass{amsart}

\usepackage[english]{babel}

\usepackage{mathrsfs, mathtools, amssymb}

\usepackage{dsfont}

\usepackage{tikz, tikz-cd}
\usepackage[paper=a4paper, margin=2cm]{geometry}

\usepackage{hyperref}
\hypersetup{
    unicode=false,          % non-Latin characters in bookmarks
    pdftoolbar=true,        % show toolbar?
    pdfmenubar=true,        % show menu?
    pdffitwindow=false,     % window fit to page when opened
    pdfstartview={FitH},    % fits the width of the page to the window
    pdftitle={Cohomology Rings of Toric bundles and the ring of conditions},    % title
    pdfauthor={Hofscheier, J.; Khovanskii, A. G.; Monin, L.},     % author
    pdfkeywords={Toric bundles} {Spherical Varieties} {Newton Polyhedra}  {Newton Okounkov bodies}, % list of keywords
    pdfnewwindow=true,      % links in new window
    colorlinks=true,       % false: boxed links; true: colored links
    linkcolor=blue,          % color of internal links
    citecolor=red,        % color of links to bibliography
    filecolor=magenta,      % color of file links
    urlcolor=cyan           % color of external links
}

\theoremstyle{plain}
\newtheorem{theorem}{Theorem}[section]
\newtheorem{lemma}[theorem]{Lemma}
\newtheorem{proposition}[theorem]{Proposition}
\newtheorem{corollary}[theorem]{Corollary}

\theoremstyle{definition}
\newtheorem{definition}[theorem]{Definition}
\newtheorem{remark}[theorem]{Remark}
\newtheorem{example}[theorem]{Example}
\newtheorem{question}[theorem]{Question}

\newcommand{\rleft}{\mathopen{}\mathclose\bgroup\left}
\newcommand{\rright}{\aftergroup\egroup\right}

\newcommand{\C}{{\mathbb{C}}}
\newcommand{\K}{{\mathbb{K}}}
\newcommand{\R}{{\mathbb{R}}}
\newcommand{\Z}{{\mathbb{Z}}}
\newcommand{\p}{{\mathbb{P}}}
\newcommand{\Q}{{\mathbb{Q}}}

\newcommand{\Cm}{{\mathcal{C}}}
\newcommand{\Fm}{{\mathcal{F}}}

\newcommand{\Om}{{\mathcal{O}}}
\newcommand{\Pm}{{\mathcal{P}}}
\newcommand{\Am}{{\mathcal{AP}}}

\newcommand{\Zm}{{\mathcal{Z}}}
\newcommand{\cL}{{\mathcal{L}}}

\newcommand{\Ann}{{\mathrm{Ann}}}
\newcommand{\BKK}{BKK }
\newcommand{\diff}{\mathop{}\!d}
\newcommand{\Diff}{{\mathrm{Diff}}}

\newcommand{\GZ}{{\mathrm{GZ}}}

\newcommand{\Hom}{{\mathrm{Hom}}}
\newcommand{\lspan}{{\operatorname{span}}}
\newcommand{\Pic}{{\mathrm{Pic}}}
\newcommand{\PP}{\mathrm{PP}}
\newcommand{\pt}{{\mathrm{pt}}}
\newcommand{\rk}{{\mathrm{rk}}}
\newcommand{\SL}{\operatorname{SL}}
\newcommand{\Sym}{{\mathrm{Sym}}}
\newcommand{\Vol}{{\mathrm{Vol}}}

\begin{document}
\selectlanguage{english}

%%%%%%%%%%%%%%%%%%%%%%%%%%%%%%%%%%%%%%%%%%%%%%%%%%%%%%%%%%%%%%%% 
\title{Cohomology Rings of Toric bundles and the ring of conditions}
\dedicatory{To the memory of Ernest Borisovich Vinberg}
\author{Johannes Hofscheier}
\address[J.\,Hofscheier]{School of Mathematical Sciences\\University of Nottingham\\Nottingham, NG7 2RD\\UK}
\email{johannes.hofscheier@nottingham.ac.uk}
\author{Askold Khovanskii}
\address[A.\,Khovanskii]{Department of Mathematics, University of Toronto, Toronto, Canada; Moscow Independent University, Moscow, Russia.}
\email{askold@math.utoronto.ca}
\author{Leonid Monin}
\address[L.\,Monin]{University of Bristol, School of Mathematics, BS8 1TW, Bristol, UK}
\email{leonid.monin@bristol.ac.uk}


\begin{abstract}
  The celebrated \BKK Theorem expresses the number of roots of a system of generic Laurent polynomials in terms of the  mixed volume of the corresponding system of Newton polytopes.
  In \cite{KP}, Pukhlikov and the second author noticed that the cohomology ring of smooth projective toric varieties over $\C$ can be computed via the \BKK Theorem.
  This complemented the known descriptions of the cohomology ring of toric varieties, like the one in terms of Stanley-Reisner algebras.

  In \cite{US03}, Sankaran and Uma generalized the ``Stanley-Reisner description'' to the case of toric bundles, i.e. equivariant compactifications of (not necessarily algebraic) torus principal bundles.
  We provide a description of the cohomology ring of toric bundles which is based on a generalization of the \BKK Theorem, and thus extends  the approach by Pukhlikov and the second author.
  Indeed, for every cohomology class of the base of the toric bundle, we obtain a BKK-type theorem.
  Furthermore, our proof relies on a description of graded-commutative algebras which satisfy Poincar\'e duality.
  
  From this computation of the cohomology ring of toric bundles, we obtain a description of the ring of conditions of horospherical homogeneous spaces as well as a version of Brion-Kazarnovskii theorem for them.
  We conclude the manuscript with a number of examples.
  In particular, we apply our results to toric bundles over a full flag variety $G/B$.
  The description that we get generalizes the corresponding description of the cohomology ring of toric varieties as well as the one of full flag varieties $G/B$ previously obtained by Kaveh in \cite{KavehVolume}. 
\end{abstract}

\maketitle

%%%%%%%%%%%%%%%%%%%%%%%%%%%%%%%%%%%%%%%%%%%%%%%%%%%%%%%%%%%%%%%%%

\section{Introduction}
\label{sec:intro}

In this paper we discuss descriptions of the cohomology ring of toric bundles and how to use them to compute the ring of conditions of horospherical homogeneous spaces.
Our main contributions are:
\begin{enumerate}
    \item A generalization of the celebrated \BKK Theorem (Theorem~\ref{BKK}). We also provide an alternative interpretation (Theorem~\ref{BKK1}). In particular, this yields a convex theoretic formula for the self-intersection polynomial on the second cohomology of a toric bundle (Theorem~\ref{thm:BKKdeg2}).
    \item A description of (not necessarily finite-dimensional) graded commutative algebras with Poincar\'{e} duality (Theorem~\ref{thm-nsdquatient}).
    \item A description of the cohomology ring of toric bundles (Theorem~\ref{cohbundle}). In particular we retrieve the early description by Uma and Sankaran (see the proof of Theorem~\ref{thm:US} given in Section~\ref{sec:cohtoricbundle}).
    \item A description of the ring of conditions of horospherical homogeneous spaces (Corollary~\ref{cor:roch}).
    \item We illustrate our results by reproving the Brion-Kazarnovskii theorem for horospherical varieties (Theorem~\ref{thm:brikaz}). We complement the illustration with the computation of the cohomology ring of some toric bundles over full flag varieties and the corresponding ring of conditions (Theorem~\ref{thm:G/U}).
\end{enumerate}

Let us start with a recollection of the classical toric case.
Let $T\simeq (\C^*)^n$ be an algebraic torus with character lattice $M$ and lattice of one-parameter subgroups $N$.
Further, let $X_\Sigma$ be a smooth projective $T$-toric variety given by a fan $\Sigma \subseteq N_\R\coloneqq N \otimes_\Z \R$.
Denote the rays of $\Sigma$ by $\Sigma(1) = \{ \rho_1,\ldots,\rho_s\}$ and their primitive generators in $N$ by $e_1,\ldots, e_s$.
Recall the following well-known description of the cohomology ring of $X_\Sigma$ (see, for instance, \cite[Theorem 12.4.1]{tor-var}):
\[
  H^*(X_\Sigma, \R)\simeq \R[x_1,\ldots, x_s]/ (I + J) \eqqcolon R_\Sigma,
\]
where $I$ is generated by monomials $x_{i_1}\cdots x_{i_t}$ such that $\rho_{i_1},\ldots, \rho_{i_t}\in \Sigma(1)$ are distinct and do not form a cone in $\Sigma$ and $J = \rleft\langle \sum_{i=1}^s \chi(e_i)x_i \colon \chi \in M \rright\rangle$.
Note that $I$ depicts the Stanley-Reisner ideal of the fan $\Sigma$ and therefore we refer to this description as the \emph{Stanley-Reisner description} of $H^*(X_\Sigma,\R)$.
This description yields an algorithm to evaluate products of cohomology classes in the top degree of $R_\Sigma$.
In Section~\ref{sec:intersection_SR}, we provide further details.

On the other hand, given line bundles $L_1, \ldots, L_t$ on $X_\Sigma$, one can directly compute a top degree intersection product $c_1(L_1)^{k_1} \cdots c_1(L_t)^{k_t}$ in $H^*(X_\Sigma,\R)$ by using the \BKK theorem \cite{Kouchnirenko} (see also \cite{Bernstein, BKK}).
Here, $c_1(L_i) \in H^2(X_\Sigma,\R)$ denotes the first Chern class of $L_i$.
More precisely, as any line bundle on $X_\Sigma$ is the difference of two ample line bundles, it suffices to evaluate products $c_1(L_1)^{k_1} \cdots c_1(L_t)^{k_t}$ with all $L_i$ ample.
By the toric dictionary, the ample line bundles $L_i$ correspond to polytopes $P_i$ whose normal fan coarsens the fan $\Sigma$.
Using, the \BKK Theorem, we obtain
\[
  c_1(L_1)^{k_1} \cdots c_1(L_t)^{k_t} = n! \cdot V(\underbrace{P_1,\ldots, P_1}_{k_1\;\text{times}}, \ldots, \underbrace{P_t, \ldots, P_t}_{k_t\;\text{times}})
\]
where $V(P_1, \ldots, P_1, \ldots, P_t, \ldots, P_t)$ denotes the mixed volume of the $n$-tuple $(P_1, \ldots, P_1, \ldots, P_t, \ldots, P_t)$.
In \cite[Section 1.4]{KP}, Pukhlikov and the second author observed that the information on these intersection products suffices to regain a description of the cohomology ring $H^*(X_\Sigma, \R)$.
As virtual polytopes play a crucial role in this description, we refer to it as the \emph{virtual polytope description} of the cohomology ring $H^*(X_\Sigma, \R)$.
In Section~\ref{sec:VP_description}, we provide further details.

In \cite{KavehVolume}, Kaveh used a similar approach to study the cohomology ring of spherical varieties provided that it is generated in degree 2.
In particular, he obtained a ``volume-polynomial-description'' for the cohomology ring of full flag varieties $G/B$ by using the relation between the volume of string polytopes and the degree of the corresponding line bundle from \cite{alexeevbrion}.
Recently, in \cite{KaveKhovanskii}, Kaveh and the second author introduced the \emph{ring of complete intersections} for an (arbitrary) algebraic variety.
For a smooth complete algebraic variety $X$ whose cohomology ring $H^*(X,\R)$ is generated in degree 2, the ring of complete intersections coincides with $H^*(X,\R)$. They also provide a description of the ring of complete intersections which generalizes the results of \cite{KavehVolume}.

Certainly, the cohomology ring of toric varieties accepts many more descriptions.
In particular, we want to emphasize the description by Brion \cite{Br96}.
Brion's description is a byproduct of the computation of the \emph{equivariant} cohomology ring.
In \cite{Br96}, Brion proves that $H^*(X_\Sigma, \R)$ is isomorphic to the ring of piecewise polynomial functions on $\Sigma$ modulo the ideal generated by global linear functions.

Our main contribution is a generalization of this picture to the case of toric bundles.
For a principal $T$-torus bundle $p \colon E \to B$ over a closed orientable real manifold $B$ and a $T$-toric variety $X_\Sigma$, let $E_\Sigma \coloneqq (E\times X_\Sigma) / T$ be the associated toric bundle (see Section~\ref{sec:general} for details).
A generalization of the Stanley-Reisner description for the cohomology ring $H^*(E_\Sigma,\R)$ was obtained by Sankaran and Uma \cite{US03}.
In \cite{roch}, the first author noticed that Brion's description also generalizes to the case of toric bundles (for more details see Section~\ref{sec:roch}).
Like in the toric case, the description by Sankaran and Uma implicitly contains an algorithm to compute products of cohomology classes in the top degree.
In Section~\ref{sec:toric_bundles}, we provide further details.

We complement the descriptions of the cohomology ring of toric bundles with a generalization of the virtual polytope description. 
Like in the toric case our description relies on (a generalization of) the \BKK theorem (see Theorem~\ref{BKK}).
More precisely, we reduce the computation of intersection numbers on the toric bundle to intersection numbers on the base.
This provides a ``BKK-type theorem'' for any choice of a cohomology class in the base $\gamma \in H^*(B, \R)$.

An equivalent description of the generalized BKK Theorem can be given as follows.
Suppose the torus $T$ has rank $n$ and the real dimension of $B$ is $k$.
Similar to the toric case, a virtual polytope $\Delta$ defines a cohomology class $\rho(\Delta)\in H^2(E_\Sigma,\R)$ on the $T$-toric bundle $p \colon E_\Sigma\to B$.
In Theorem~\ref{BKK1}, for any given $j$, we define a map which associates to a virtual polytope $\Delta$ a cohomology class $\gamma_{2j}(\Delta)\in H^{2j}(B,\R)$ such that
\[
  \rho(\Delta)^{n+j} \cdot p^*(\gamma) = \gamma_{2j}(\Delta) \cdot \gamma \text{,}
\]
for any $\gamma\in H^{k-2j}(B,\R)$.
Here ``$\cdot$'' denotes the cup product on the respective cohomology ring.
We call the class $\gamma_{2j}(\Delta)$ \emph{the horizontal part} of $\rho(\Delta)^{n+j}$. 
We refer to Section~\ref{sec:bkk-tb} for further details.

In particular, our version of the BKK theorem allows us to compute the self-intersection polynomial on $H^2(E_\Sigma,\R)$ (see Theorem~\ref{thm:BKKdeg2}), that is the map
\[
H^2(E_\Sigma,\R) \to \R, \quad \gamma\mapsto \int_{E_\Sigma} \gamma^{n+k/2} \qquad \text{if $k$ is even and zero otherwise.}
\]
It should be mentioned that our statements above are true not only in the algebraic category, but more generally hold for smooth manifolds.

Furthermore, unlike the toric case or the case studied in \cite{KavehVolume}, the cohomology ring of $E_\Sigma$ is not generated in degree $2$ in general.
However, graded commutative algebras which satisfy Poincar\'e duality accept a description generalizing the one used by Pukhlikov and the second author (see Theorem~\ref{thm-nsdquatient}).

From the generalized BKK-theorem and our description of graded commutative algebras with Poincar\'e duality, we obtain a description of the cohomology ring of toric bundles (see Theorem~\ref{cohbundle}).
In Section~\ref{sec:cohtoricbundle} we show that our approach naturally yields an alternative description of the cohomology ring of toric bundles obtained by Uma and Sankaran \cite{US03}. 

Our description is well suited to compute cohomology rings of toric bundles over a fixed base manifold $B$.
In particular, a computation of the ring of conditions of horospherical homogeneous spaces naturally follows (Corollary~\ref{cor:roch}).
Recall that the ring of conditions is an intersection ring for (not necessarily complete) homogeneous spaces. 
Furthermore, for a connected complex reductive group $G$ a homogeneous space $G/H$ is called \emph{horospherical} if $H$ is a closed subgroup in $G$ containing a maximal unipotent subgroup. 

We conclude the manuscript by applying our results to the case of toric bundles over generalized flag varieties $G/P$ (also known as toroidal horospherical varieties). In Theorem~\ref{thm:brikaz} we obtain a version of the Brion-Kazarnovskii theorem (\cite{bri89,kaz87}) for our setting. In particular, we recover the results of \cite{KaveKhovanskii} (and thus the results of \cite{KavehVolume}). We continue with the case of toric bundles over full flag varieties $G/B$ for which we give two descriptions of the cohomology rings and the ring of conditions (Theorems~\ref{thm:G/HWeyl} and~\ref{thm:G/U}). The latter description agrees with a construction of Newton-Okounkov bodies from \cite{KKhor}. Indeed, we introduce the notion of \emph{fibered virtual polytopes} which is the extension of the family of polytopes accepting a projection onto a polytope in the positive Weyl chamber with fibers being string polytopes (see Sections~\ref{sec:string-poly} and~\ref{sec:tb-over-flag}).

%%%%%%%%%%%%%%%%%%%%%%%%%%%%%%%%%%%%%%%%%%%%%%%%%%%%%%%%%%%%%%%% 

\subsection*{Organization of the paper}
Section~\ref{sec:cohomology_toric} provides further details for the two descriptions of the cohomology ring of toric varieties mentioned above.
Section~\ref{sec:general} collects results about the cohomology ring of toric bundles and horospherical varieties.
Section~\ref{sec:bkk-tb} contains the statement of the generalized BKK theorem.
Section~\ref{sec:convex-chains} provides further details about virtual polytopes and the related convex chains.
Section~\ref{BKKproof} contains the proof of the generalized BKK theorem.
Section~\ref{commalg} studies graded commutative finite-dimensional algebras satisfying Poincar\'e duality.
Section~\ref{sec:applications} describes the cohomology ring of toric bundles (in general) and computes the ring of conditions of horospherical homogeneous spaces.
Section~\ref{sec:base-var} discusses a calculation of the self-intersection polynomial on the second cohomology group of toric bundles yielding an easier description of the cohomology ring provided it's generated in degree 2.
Section~\ref{sec:examples} contains a version of Brion-Kazarnovskii theorem and computes the cohomology ring of certain classes of toric bundles including toric bundles over full flag varieties $G/B$ and projective bundles.

%%%%%%%%%%%%%%%%%%%%%%%%%%%%%%%%%%%%%%%%%%%%%%%%%%%%%%%%%%%%%%%% 

\subsection*{Acknowledgements}

We thank Megumi Harada and Kiumars Kaveh for several helpful and inspiring conversations. We also would like to thank Michel Brion for his comments on the earlier version of this paper.
The first author is supported by a Nottingham Research Fellowship from the University of Nottingham.
The second author is partially supported by the Canadian Grant No.~156833-17.
The third author is supported by EPSRC Early Career Fellowship EP/R023379/1.
 
%%%%%%%%%%%%%%%%%%%%%%%%%%%%%%%%%%%%%%%%%%%%%%%%%%%%%%%%%%%%%%%%%

\section{The Cohomology Ring of a Toric Variety}
\label{sec:cohomology_toric}

We assume fundamental knowledge of toric geometry and refer to \cite{tor-var} for further details and references.
For the reader's convenience, we give further details of the two descriptions of the cohomology ring of toric varieties mentioned in Section~\ref{sec:intro}.
We continue to use the notation from the Introduction.
In particular, $M$ denotes a lattice, i.e., a finitely generated free abelian group, i.e., $M \simeq \Z^n$, $N = \Hom_\Z(M, \Z)$ its dual lattice, and $\langle \cdot, \cdot \rangle \cdot M \times N \to \Z$ their dual pairing.
Suppose that $X_\Sigma$ is a smooth projective toric variety given by a smooth projective fan $\Sigma \subseteq N_\R \coloneqq N\otimes \R$.

%%%%%%%%%%%%%%%%%%%%%%%%%%%%%%%%%%%%%%%%%%%%%%%%%%%%%%%%%%%%%%%% 

\subsection{Intersection products in the Stanley-Reisner description}
\label{sec:intersection_SR}

The Stanley-Reisner description of the cohomology ring yields an algorithm to compute products of cohomology classes in the top degree of $H^*(X_\Sigma,\R)$:
As $R_\Sigma$ is generated in degree 1, it suffices to consider monomials $x_{i_1}^{k_1}\ldots x_{i_t}^{k_t}$.
Recall that the graded piece of top degree of $R_\Sigma$, is one-dimensional and generated by $x_{i_1}\cdots x_{i_n}$ for any collection $i_1,\ldots, i_n$ of indices such that the rays $\rho_{i_1}, \ldots, \rho_{i_n}$ generate a full-dimensional cone in $\Sigma$.
Indeed, all such monomials $x_{i_1}\cdots x_{i_n}$ yield the same element in $R_\Sigma$.
Therefore, the evaluation of a monomial $x_{i_1}^{k_1}\ldots x_{i_t}^{k_t}$ in the top degree amounts to expressing it as a linear combination of square free monomials.

For a monomial $x_{i_1}^{k_1}\cdots x_{i_t}^{k_t}$ let $\sum_{i=1}^t(k_i -1)$ be its \emph{multiplicity}, so that being square free is equivalent to having multiplicity $0$.
To simplify notation, consider the monomial $x_1^{k_1}\cdots x_t^{k_t}$ (always possible by reordering the variables).
Suppose $x_1^{k_1}\cdots x_t^{k_t}$ is a monomial with multiplicity $m>0$ and $k_1>1$.
The goal is to express it as a linear combination of monomials of smaller multiplicity.
If $\rho_1,\ldots, \rho_t$ do not form a cone, then the monomial is equal to zero in $R_\Sigma$ and we are done.
Otherwise, the set $\{e_1,\ldots, e_t\}$ can be extended to a lattice basis of $N$, and thus there is $\chi \in M$ such that $\chi(e_1)=1$ and  $\chi(e_j)=0$, for $j=2,\ldots,t$.
Note that $\chi$ induces a linear relation in $J$, so that we obtain
\[
  x_1^{k_1}\ldots x_t^{k_t} = x_1^{k_1-1} \ldots x_t^{k_t} \cdot \rleft(-\sum_{k=t+1}^s \chi(e_k)x_k\rright) \text{.}
\]
This is a linear combination of monomials of multiplicity $m-1$.
Applying this procedure inductively, we end up with a linear combination of square free monomials, and thus obtain an evaluation of $x_1^{k_1} \cdots x_t^{k_t}$.

%%%%%%%%%%%%%%%%%%%%%%%%%%%%%%%%%%%%%%%%%%%%%%%%%%%%%%%%%%%%%%%%%

\subsection{The virtual polytope description}
\label{sec:VP_description}

As noted in the introduction, an arbitrary line bundle on $X_\Sigma$ can be expressed as a differences of ample line bundles which, by the toric dictionary, correspond to polytopes whose normal fan coarsens $\Sigma$.
Recall that a fan $\Sigma'$ in $N_\R$ \emph{coarsens} $\Sigma$ if every cone of $\Sigma$ is contained in a cone of $\Sigma'$ and $|\Sigma'| = |\Sigma|$.
One also says that $\Sigma$ \emph{refines} $\Sigma'$.
In particular, every cone of $\Sigma'$ is a union of cones of $\Sigma$.
One of the crucial ideas in \cite{KP} is to reformulate this observation on line bundles into convex geometry.
Therefore let us recall the definition of the space of virtual polytopes:

Let $\Pm_\Sigma^+$ be the set of polytopes in $M_\R \coloneqq M \otimes \R$ such that $\Sigma$ refines their normal fans.
Recall that the \emph{Minkowski sum} of two polytope $P,Q$ in $M_\R$ is given by
\[
  P + Q = \rleft\{ p + q \colon p \in P, q \in Q \rright\} \text{,}
\]
and that the normal fan of a Minkowski sum is the coarsest common refinement of the individual fans (see \cite[Proposition 7.12]{Ziegler}).
Thus the Minkowski sum endows the structure of a monoid on $\Pm_\Sigma^+$.
It is well-known that a commutative semigroup can be embedded in a group if the operation is \emph{cancellative}, i.e., $P_1 + Q = P_2 + Q$ implies $P_1 = P_2$.
It is straightforward to show that the Minkowski sum satisfies this property.
We denote the smallest group $\Pm_\Sigma$ containing the monoid $\Pm_\Sigma^+$ the group of \emph{virtual polytopes} (associated to the fan $\Sigma$).
Clearly $\Pm_\Sigma^+$ accepts a natural multiplication map by the nonnegative reals $\R_{\ge0}$ which straightforwardly extends to a scalar multiplication by $\R$ on the group of virtual polytopes, thus turning $\Pm_\Sigma$ into a real vector space.
Moreover, note that this also shows that $\Pm_\Sigma^+$ is a cone in $\Pm_\Sigma$, and as it generates the ambient vector space it is full-dimensional.

For us the following alternative description of this vector space will be useful.
Recall that one associates to a non-empty compact convex set $A \subseteq M_\R$ its \emph{support function}:
\[
  h_A \colon N_\R \to \R; \quad x \mapsto h_A(x) \coloneqq \sup \rleft\{ \langle a, x \rangle \colon a \in A \rright\}
\]
In other words, given an exterior normal vector $x\neq0$, the position of the corresponding supporting hyperplane is determined by its support function $h_A$, namely $A \subseteq \{ y \in M_\R \colon \langle y, x \rangle \le h_A(u) \}$ with the hyperplane $\{ y \in M_\R \colon \langle y, x \rangle = h_A(u) \}$ intersecting $A$ in a non-empty face.
The crucial observation is that $h_A$ uniquely describes $A$.
Indeed, one can show that $A = \{ x \in M_\R \colon \langle x, u\rangle \leq h_A(u)\; \text{for all} \; u \in N_\R \}$.

Using the language of support functions $\Pm_\Sigma^+$ and $\Pm_\Sigma$ can be reformulated as follows:
\begin{align*}
  \Pm_\Sigma^+ &= \rleft\{ h \colon N_\R \to \R \colon \text{convex, piecewise linear function on} \; \Sigma  \rright\}\text{,}\\
  \Pm_\Sigma &= \rleft\{ h \colon N_\R \to \R \colon \text{piecewise linear function on} \; \Sigma \rright\} \text{,}
\end{align*}
where a function $h \colon N_\R \to \R$ is piecewise linear on $\Sigma$ if $h_\sigma \coloneqq h|_\sigma$ can be extended to a linear function on $\lspan\{ \sigma \}$ for any $\sigma \in \Sigma$.
In addition, such a function is \emph{convex} if $h(u) \ge h_\sigma(u)$ for any $\sigma \in \Sigma$ and $u \in \lspan\{\sigma\}$.

With the above observation on line bundles we have that $\Pm_\Sigma/M_\R \simeq H^2(X_\Sigma, \R)$ where $M_\R$ is interpreted as a linear subspace of $\Pm_\Sigma$, namely the subspace of linear functions.
Indeed, any $m \in M_\R$ can be expressed as a difference $(P+m)-P$ for any polytope $P \in \Pm_\Sigma^+$.
Note that the volume yields a well-defined polynomial function on $\Pm_\Sigma/M_\R$.
The following theorem is implicit in \cite{KP} and forms a key idea to get a description of the cohomology ring $H^*(X_\Sigma, \R)$.
The following version is taken from \cite[Theorem 1.1]{KavehVolume}.

\begin{theorem}
  \label{PKh}
  Suppose $A = \bigoplus_{i=0}^n A_i$ is a graded finite dimensional commutative algebra over a field $\K$ of characteristic $0$ such that $A$ is generated (as an algebra) by the elements $A_1$ of degree one, $A_0 \simeq A_n \simeq \K$, and the bilinear map $A_i \times A_{n-i} \to A_n$ is non-degenerate for any $i = 0, \ldots, n$ (Poincar\'{e} duality).
  Then 
  \[
    A\simeq \K[t_1, \ldots, t_r]/\{p(t_1, \ldots, t_r) \in \K[t_1, \ldots, t_r] \colon p(\tfrac{\partial}{\partial x_1}, \ldots, \tfrac{\partial}{\partial x_r}) f(x_1, \ldots, x_r) = 0 \}
  \]
  where we identify $A_1$ with $\K^r$ via a basis $v_1, \ldots, v_r$ and define $f \colon A_1 \simeq \K^r \to \K$ as the polynomial given by $f(x_1,\ldots, x_r) = (x_1v_1 + \ldots + x_rv_r)^n \in A_n \simeq k$.
\end{theorem}

\begin{remark}
  \label{rem:diff-ops}
  Instead of identifying $A_1$ with $\K^r$, the previous theorem accepts a basis-free formulation in terms of $\Diff(A_1)$, the ring of differential operators with constant coefficients:
  $A \simeq \Diff(A_1) / \{ p \in \Diff(A_1) \colon p (f) = 0\}$.

  In Section~\ref{commalg}, we study algebras with Poincar\'e duality which are not necessarily generated in degree $1$.
  In particular, we prove Theorem~\ref{thm-nsdquatient} which is a generalization of Theorem~\ref{PKh}.
  Then in Theorem~\ref{thm-comm-alg}, we show that Theorem~\ref{PKh} is indeed a corollary of Theorem~\ref{thm-nsdquatient}.
\end{remark}

Under the identification $\Pm_\Sigma/M_\R \simeq H^2(X_\Sigma, \R)$ together with \BKK theorem, we get that the function $H^2(X_\Sigma,\R) \to H^{2n}(X_\Sigma, \R) \simeq \R; x \mapsto x^n$ corresponds to the volume polynomial on $\Pm_\Sigma/M$, and thus we obtain (using the formulation from Remark~\ref{rem:diff-ops}):
\[
  H^*(X_\Sigma, \R) \simeq \Diff(\Pm_\Sigma /M_\R ) / \Ann(\Vol) \simeq \Diff(\Pm_\Sigma) / \Ann(\Vol)
\]
where $\Ann(\Vol)$ denotes the ideal of differential operators which annihilate the volume polynomial $\Vol \colon \Pm_\Sigma \to \R$.

%%%%%%%%%%%%%%%%%%%%%%%%%%%%%%%%%%%%%%%%%%%%%%%%%%%%%%%%%%%%%%%%%

\section{Toric bundles and horospherical varieties}
\label{sec:general}

In this section, we collect several facts on toric bundles which will be used below.
We also introduce horospherical varieties which yield an important class of toric bundles.

%%%%%%%%%%%%%%%%%%%%%%%%%%%%%%%%%%%%%%%%%%%%%%%%%%%%%%%%%%%%%%%% 

\subsection{Toric bundles}
\label{sec:toric_bundles}

Let $G$ be a topological group and let $p \colon E \to B$ be a \emph{$G$-principal bundle} over a topological space $B$, i.e., $p \colon E \to B$ is a fiber bundle with structure group $G$ equipped with a $G$-atlas such that the action of $G$ on the fibers $p^{-1}(b)$ for any $b \in B$ is free and transitive.
For convenience, we will identify the fibers of $p$ with $G$, so that we obtain a (right) action of $G$ on $E$ which preserves the fibers and is transitive and free on them.
To any $G$-principal bundle $p \colon E \to B$ and any topological space $X$ equipped with a continuous action by $G$, one associates a fiber bundle by introducing a (right) action on the product $E \times X$:
\[
  (e, x) \cdot g \coloneqq (e\cdot g, g^{-1}\cdot x) \text{.}
\]
The \emph{associated fiber bundle} is given as the quotient $E \times_G X \coloneqq (E \times X) / G$.
It is a fiber bundle with fiber $X$.
If $G = T$ is an algebraic torus, then a $T$--principal bundle is also called a \emph{torus bundle}.

Crucial to the understanding of the cohomology of fiber bundles is the following theorem (see, for instance, \cite[Theorem 5.11]{BottTu} or \cite[Theorem 4D.1]{Hatcher}):
\begin{theorem}[Leray--Hirsch]
  \label{thm:Leray-Hirsch}
  Let $E$ be a fibre bundle with fibre $F$ over a compact manifold $B$.
  If there are global cohomology classes $u_1,\ldots , u_r$ on $E$ whose restrictions $i^*(u_i)$ form a basis for the cohomology of each fibre $F$ (where $i \colon F \to E$ is the inclusion), then we have an isomorphism of vector spaces:
  \[
    H^*(B, \R)\otimes H^*(F, \R) \to H^*(E, \R); \quad \sum_{i,j} b_i \otimes i^*(u_j) \mapsto \sum_{i,j} p^*(b_i) \cdot u_j\text{.}
  \]
\end{theorem}

\begin{corollary}
  \label{cor:cohmodule}
  If $T$ is an algebraic torus, $p \colon E \to B$ a $T$-toric bundle as in Theorem~\ref{thm:Leray-Hirsch}, and $X$ a smooth projective $T$--toric variety, then as a group the cohomology of $E_X = E \times_T X$ is given by
  \[
    H^*(E_X, \R) \simeq H^*(B, \R) \otimes H^*(X, \R) \text{.}
  \]
\end{corollary}
\begin{proof}
  As a group, the cohomology of smooth projective toric varieties is generated by classes Poincar\'{e} dual to orbit closures.
  For any orbit closure $\overline{O}$ of $X$, let $E_{\overline{O}} = E \times_T \overline{O}$ be a submanifold of $E_X$.
  Note that $E_X$ is compact, and thus Poincar\'{e} duality applies.
  The cohomology classes $u_1, \ldots, u_r$ Poincar\'{e} dual to these submanifolds satisfy the condition of Theorem~\ref{thm:Leray-Hirsch}.
  The statement follows.
\end{proof}

Corollary~\ref{cor:cohmodule} yields a description of the cohomology \emph{group} of $E_X$.
Crucial for this description is the map which associates to any orbit closure $\overline{O}$ of $X$ the Poincar\'{e} dual of the corresponding $T$-invariant submanifold of $E_X$.
By restricting this map to divisors and extending using linearity, we obtain
\[
  \rho\colon \Pm_\Sigma \to H^2(E_X, \R) \qquad \text{(where $\Sigma$ is the fan of $X = X_\Sigma$)}
\]
which plays an important role in our description of the cohomology ring of $E_X$.
We provide more details for $\rho$.
Recall that $\rho_1,\ldots,\rho_r$ denote the rays of $\Sigma$ with integer generators $e_1,\ldots,e_r$.
Let $D_1,\ldots,D_r$ be the corresponding divisors in $X$.
We also write $D_i$ for the submanifold $E \times_T D_i$ of $E_X$.
Then for $\Delta \in \Pm_\Sigma$, we have:
\[
  \rho(\Delta) = \sum_{i=1}^r h_\Delta(e_i) [D_i] \in  H^2(E_X, \R),
\]
where $[D_i]$ is the class Poincar\'{e} dual to $D_i\subseteq E_X$ and $h_\Delta \colon N_\R \to \R$ is the support function of $\Delta$.

The following observation about $\rho(\cdot)$ will be crucial for our approach.
Any character $\lambda \in M$ defines a one--dimensional representation $\C_\lambda$ of $T$, namely $t \cdot z = \lambda(t) z$ for $t \in T$, and $z \in \C_\lambda$.
If $\cL_\lambda$ denotes the associated complex line bundle on $B$, i.e. $\cL_\lambda\simeq E\times_T\C_\lambda$, then $\cL_{\lambda+\mu}=\cL_\lambda\otimes\cL_\mu$, and thus we obtain a group homomorphism:
\[
  c\colon M\to H^2(B,\Z), \quad \lambda \mapsto c_1(\cL_\lambda),
\]
where $c_1(\cL_\lambda)$ is the first Chern class.
By linearity, we extend the homomorphism to a map of vector spaces:
\[
  c\colon  M_\R \to H^2(B,\R).
\]
\begin{proposition}\label{cherneq}
  Let $\Sigma$ be a smooth complete fan, and $p\colon E_X \to B$ as before.
  Then for any character $\lambda \in M$:
  \[
    p^* c(\lambda)= \rho(\lambda),
  \]
  where on the right hand side of the equality $\lambda$ is regarded as a virtual polytope.
\end{proposition}
\begin{proof}
  A character $\lambda$ of $T$ defines a $T$-invariant divisor $\mathrm{div}(\lambda)$ on $X$ which can be expressed as $\mathrm{div}(\lambda) = \sum_{i=1}^r k_i D_i$ where the $D_i \subseteq X$ are distinct irreducible divisors and $k_i \in \Z$.
  Every divisor $D_i$ comes with a (naturally) linearized line bundle $\cL_i$ and a global regular section $s_i$ such that $\mathrm{div}(s_i) = D_i$.
  The sections $s_i$ yield generic global (smooth) sections of the associated line bundle $E \times_T \cL_i \to E \times_T X$ with degeneracy locus $E \times_T D_i$.
  By \cite[Gauss--Bonnet Formula II, p.~413]{GriffithsHarris}, $c_1(E\times_T \cL_i)$ is Poincar\'{e} dual to $E\times_T D_i$.
  Using linearity of $c_1$, we get $c_1(E\times_T \cL_\lambda) = \rho(\lambda)$.
  It is straightforward to verify that $E \times_T \cL_\lambda=p^* \cL_\lambda$, and so
  \[
    p^*c(\lambda)=c_1(p^*\cL_\lambda) = c_1(E \times_T \cL_\lambda) = \rho(\lambda) \text{.} \qedhere
  \]
\end{proof}

The map of lattices $c\colon M\to H^2(B,\Z)$ is an invariant that uniquely describes torus bundles in the topological category.
Indeed, a torus bundle $E$ can be recovered from $ c\colon M \to H^2(B,\Z)$ as follows.

Choose a basis $e_1, \ldots, e_k$ of $M$.
Since (topological) complex line bundles are classified by their Chern classes, there exist bundles $\cL_i$ on $B$ uniquely determined by $c_1(\cL_i) = c(e_i)\in H^2(B,\Z)$ for $i=1, \ldots, k$.
The rest of the construction follows a general recipe for how to get a torus bundles $E\to B$ from a family of (topological) complex line bundles $\cL_1, \ldots, \cL_k$ on $B$:

Consider the vector bundle $\cL_1\oplus\ldots\oplus\cL_k$ of rank $k$ over $B$.
Note for any subset $\{i_1, \ldots, i_s\} \subseteq \{ 1, \ldots, k\}$ the vector bundle $\cL_{i_1}\oplus\ldots\oplus\cL_{i_s}$ of rank $s<k$ naturally embeds into $\cL_1\oplus\ldots\oplus\cL_k$.
Let $E\subseteq \cL_1\oplus\ldots\oplus\cL_k$ be the complement of the coordinate vector subbundles $\cL_{i_1}\oplus\ldots\oplus\cL_{i_s}$.
The algebraic torus $(\C^*)^k$ acts freely on $E$ via coordinate-wise scaling, and $E\to B$ is a (topological) principal  torus bundle.
We arrive at the following result.

\begin{proposition}
  \label{prop:toricbun}
  Let $B$ be a closed oriented manifold and $T$ be an algebraic torus with character lattice $M(T)$.
  Then (topological) $T$-principal  bundles over $B$ are in bijection with homomorphisms $c\colon M(T)\to H^2(B,\Z)$.
\end{proposition}

If the base manifold $B$ is an algebraic variety, there is an analogous description of algebraic $T$-principal  bundles over $B$. In this case one has to consider the homomorphism $\widetilde c: M(T) \to \Pic(B)$ given by 
\[
    \widetilde c \colon M(T) \to \Pic(B); \; \lambda \mapsto \cL_\lambda \in \Pic(B) \text{.}
\]

We conclude this section with the following description of the cohomology ring of toric bundles by Sankaran and Uma.
We note that this description is true in greater generally as stated here and we refer the reader to \cite[Theorem 1.2]{US03} for the complete statement.
The following version will be sufficient for what follows. 
\begin{theorem}
  \label{thm:US}
  We continue to use the notation from above.
  If $X_\Sigma$ is smooth and projective, the cohomology ring $H^*(E_\Sigma,\R)$ is isomorphic (as an $H^*(B, \R)$-algebra) to the quotient of $H^*(B,\R)[x_1, \ldots, x_r]$ by
  \[
     \rleft( \left\langle x_{j_1} \cdots x_{j_k} \colon \rho_{j_1}, \ldots, \rho_{j_k} \; \text{do not span a cone of} \; \Sigma \right\rangle + \left\langle c\left( \lambda \right) - \sum_{i=1}^n \langle e_i, \lambda \rangle x_i \colon \lambda \in M \right\rangle \rright) \text{.}
  \]
\end{theorem}
Note the similarities with the Stanley-Reisner description of the cohomology ring of toric varieties.
Indeed, the first ideal in Theorem~\ref{thm:US} corresponds to the Stanley-Reisner ideal of the corresponding toric variety.
In particular, the algorithm from Section~\ref{sec:intersection_SR} can be used to compute products in the top degree of $H^*(E_X, \R)$. In the end of Section~\ref{sec:cohtoricbundle} we will give an alternative proof of Theorem~\ref{thm:US} based on our version of \BKK Theorem.

%%%%%%%%%%%%%%%%%%%%%%%%%%%%%%%%%%%%%%%%%%%%%%%%%%%%%%%%%%%%%%%% 

\subsection{Horospherical varieties}
\label{sec:horospherical}

In this section, we recall horospherical varieties which provide an important source of examples of toric bundles.
We conclude with a description of their associated cohomology ring.

Let $G$ be a connected reductive complex algebraic group.
A closed subgroup $H \subseteq G$ is called \emph{horospherical} if it contains a maximal unipotent subgroup $U$ of $G$ and $E \coloneqq G/H$ is called a \emph{horospherical homogeneous space}.
Let $B \subseteq G$ be a Borel subgroup whose unipotent radical is $U$.
It is known that $P \coloneqq N_G(H)$ is a parabolic subgroup containing $B$.
The natural quotient $p \colon E \to G/P$ is a torus bundle with fibre the torus $T = P/H$.
Let $M$ be the character lattice of $T$.
Let $\Sigma \subseteq N_\R \coloneqq \Hom_\Z( M, \R)$ be a fan with corresponding toric variety $X = X_\Sigma$.
The associated toric bundle $p \colon E_X \to G/P$ is a horospherical variety.
In general, an irreducible normal $G$-variety $Y$ together with an open equivariant embedding $G/H \hookrightarrow Y$ is called a \emph{horospherical variety}.
Recall that a morphism of $G$-varieties $\varphi \colon Y \to Y'$ is called \emph{equivariant} if $\varphi(g \cdot y) = g \cdot \varphi(y)$ for any $g \in G$ and $y \in Y$.
We refer to \cite{Knop:SphEmbd, Timashev:HSEE} for further details and references on horospherical varieties.

We conclude this section with a combinatorial description of the cohomology ring $H^*(E_X,\R)$.
Recall from the previous section that any character $\lambda \in M$ gives rise to a line bundle $\cL_\lambda$ on $G/P$, so that we obtain a group homomorphism $c \colon M \to H^2(G/P, \Z)$.
The following statement combinatorially describes the cohomology ring of $E_X$ provided $X_\Sigma$ is smooth and projective.
It is a special case of a more general result.
\begin{theorem}[{\cite[Theorem 1.2]{US03}}]
  \label{thm:cr-hs}
  Suppose $\Sigma$ has rays $\rho_1, \ldots, \rho_r$ with primitive vectors $v_1, \ldots, v_n \in N \coloneqq \Hom_\Z(M, \Z)$ along the edges $\rho_i$.
  If $X_\Sigma$ is smooth and projective, then the cohomology ring $H^*(E_X,\R)$ is isomorphic as an $H^*(G/P, \R)$-algebra to the quotient of $H^*(G/P,\R)[x_1, \ldots, x_r]$ by the sum of ideals
  \[
    \left\langle x_{j_1} \cdots x_{j_k} \colon \rho_{j_1}, \ldots, \rho_{j_k} \; \text{do not span a cone of} \; \Sigma \right\rangle + \left\langle c(\lambda) - \sum_{i=1}^r \langle v_i, \lambda \rangle x_i \colon \lambda \in M \right\rangle \text{.}
  \]
\end{theorem}

%%%%%%%%%%%%%%%%%%%%%%%%%%%%%%%%%%%%%%%%%%%%%%%%%%%%%%%%%%%%%%%%%

\section{\BKK theorems for toric bundles}
\label{sec:bkk-tb}

This section provides an overview of the main results of this paper: Theorems~\ref{BKK} and~\ref{BKK1}.
Details and proofs will be given in Section~\ref{BKKproof}, and in Section~\ref{sec:applications} these theorems will be used to describe the cohomology ring of toric bundles.

As before, let $p\colon E\to B$ be a principal torus bundle with respect to a torus $T\simeq (\C^*)^n$ over a compact smooth orientable manifold $B$ of real dimension $k$.
Let $M$ be the character lattice of $T$ and $\Sigma \subseteq N_\R = \Hom_\Z(M,\R)$ be a smooth complete fan which gives rise to a toric variety $X = X_\Sigma$.
Let $E_X$ be the total space of the associated toric bundle.
Note that $E_X$ is a compact smooth orientable manifold of real dimension $k+2n$.
To keep notation simple we denote the projection map of the toric bundle by $p\colon E_X \to B$ as well.

Our main theorems show that a choice of a natural number $i\leq\tfrac{k}{2}$ and $\gamma\in H^{k-2i}(B, \R)$ gives rise to a \mbox{\BKK-type} theorem.
First, we define two functions $I_\gamma$ and $F_\gamma$ on $\Pm_\Sigma$ as follows.

Let $f_\gamma\colon M_\R \to \R$ be given by 
\[
  f_\gamma(x)=c(x)^{i} \cdot \gamma\text{,}
\]
where ``$\cdot$'' denotes the cup product of the cohomology ring $H^*(B,\R)$.
Here we identified the top cohomology group $H^k(B,\R)$ with $\R$ by pairing with the fundamental class of $B$.
Since $c\colon M_\R\to H^2(B,\R)$ is a linear map, $f_\gamma$ is a homogeneous polynomial of degree $i$ on $M_\R$.
This leads to the definition of $I_\gamma$:
\[
  I_\gamma\colon \Pm_\Sigma\to \R; \quad I_\gamma(\Delta) \coloneqq \int_\Delta f_\gamma(x) \diff\mu  \quad \text{for} \; \Delta\in \Pm^+_\Sigma,
\]
where $\mu$ denotes the Lebesgue measure on $M_\R$ normalized with respect to the lattice $M$, i.e. a cube spanned by an affine lattice basis of $M$ has volume $1$.
Note that we gave an explicit formula for $I_\gamma$ on the full-dimensional cone $\Pm^+_\Sigma$.
By Theorem~\ref{ultimatepoly}, $I_\gamma$ extends to a homogeneous polynomial of degree $n+i$ on $\Pm_\Sigma$.

Recall the definition of $\rho\colon \Pm_\Sigma\to H^{2}(E_X, \R)$ from Section~\ref{sec:toric_bundles}.
This leads to the definition of the function $F_\gamma$:
\[
  F_\gamma\colon \Pm_\Sigma\to \R; \quad F_\gamma(\Delta) \coloneqq \rho(\Delta)^{n+i}\cdot p^*(\gamma).
\]
Clearly, $F_\gamma$ is a homogeneous polynomial of degree $n+i$ on $\Pm_\Sigma$.

The main result of this section is the following analog of the \BKK theorem for toric bundles.
Indeed, it expresses certain intersection numbers of cohomology classes as mixed integrals.

\begin{theorem}
  \label{BKK}
  The polynomials $I_\gamma$ and $F_\gamma$ are proportional with coefficient of proportionality given by:
  \[
    (n+i)!\cdot I_\gamma(\Delta)=i!\cdot F_\gamma(\Delta) \qquad \text{for any $\Delta\in \Pm_\Sigma$.}
  \]
  In particular, the polarizations of $I_\gamma$ and $F_\gamma$ are proportional multilinear forms, i.e. for any $\Delta_1,\ldots,\Delta_{n+i}\in \Pm_\Sigma$
  \[
    (n+i)!\cdot I_\gamma(\Delta_1,\ldots,\Delta_{n+i})=i!\cdot F_\gamma(\Delta_1,\ldots,\Delta_{n+i}).
  \]
\end{theorem}

For the reader's convenience, we recall the concept of polarization (or equivalently mixed integrals).
Let $V$ be a (possibly infinite-dimensional) vector space.
Here $V = \Pm_\Sigma$ is finite-dimensional, however eventually we want to consider $\Pm$ which is an infinite-dimensional vector space.

\begin{definition}
  \label{def:Lie}
  The \emph{Lie derivative} $L_v f(x)$ of a function $f \colon V \to \R$ with respect to a vector $v \in V$ at a point $x \in V$ is the limit
  \[
    L_v f(x) = \lim_{t \to 0} \frac{f(x+t\cdot v) - f(x)}{t} \text{,}
  \]
  provided this limit exists.
\end{definition}

Note that for a finite-dimensional vector space $V$ the concept of polynomials $f \colon V \to \R$ is well-known.
In the next section, we recall how to extend the concept of polynomials to the infinite-dimensional case.
As the concept of polarization works both in the finite- and the infinite-dimensional case, we provide details for the general situation here.
The careful reader may want to consult Definition~\ref{def:poly} before reading further.

Recall that for a homogeneous polynomial $f\colon V \to \R$ of degree $m$, the \emph{polarization of $f$} is the unique symmetric multilinear form $g\colon V^m \to \R$ such that $g(v,\ldots,v)=f(v)$. 
It is a well-known fact that for any vector space $V$ and any homogeneous polynomial $f$ of degree $m$, the polarization exists and can be defined as follows:
\begin{equation}\label{eq:polarisation_formula}
  g(v_1,\ldots,v_m) = \frac{1}{m!}L_{v_1}\ldots L_{v_m}f.
\end{equation}
For a homogeneous polynomial $f\colon M_\R\to \R$, the polarization of $I_f$ is called the \emph{mixed integral of $f$}.

Note that in the case of a trivial torus bundle $p:T\to \pt$, Theorem~\ref{BKK} reduces to the classical BKK theorem.
We conclude this section with an alternative interpretation of Theorem~\ref{BKK} which can be favourable for certain applications.
By the Leray-Hirsch Theorem (see Corollary~\ref{cor:cohmodule}), $H^*(E_X, \R) \simeq H^*(B, \R)\otimes H^*(X, \R)$, and so for $\Delta \in \Pm_\Sigma$, the cycle $\rho(\Delta)^{n+i}\in H^{2n+2i}(E_X, \R)$ can be written as
\[
  \rho(\Delta)^{n+i} =  b_{2n+2i}\otimes x_{0} +  b_{2n+2i-2}\otimes x_{2} +\ldots + b_{2i+2}\otimes x_{2n-2} +  b_{2i}\otimes x_{2n},
\]
with $b_{s}\in H^s(B, \R)$ and $x_{r}\in H^{r}(X, \R)$.
As $X$ is a smooth, complete toric variety, its cohomology groups in odd degrees vanish, and therefore $x_{2k+1}=0$ for any $k$.
If $x_{2n}$ is normalized such that it is dual to a point, we call $b_{2i}$ the \emph{horizontal part} of $\rho(\Delta)^{n+i}$.
Equivalently, the horizontal part $b_{2i}$ of $\rho(\Delta)^{n+i}$ is the unique class in $H^{2i}(B,\R)$ such that
\[
  \rho(\Delta)^{n+i}\cdot p^*(\eta) = b_{2i} \cdot \eta \text{,}
\]
for any $\eta\in H^{k-2i}(B,\R)$.
Then Theorem~\ref{BKK} accepts the following reformulation.

\begin{theorem}
  \label{BKK1}
  For any $\Delta \in \Pm_\Sigma$, the horizontal part of $\rho(\Delta)^{n+i}$ can be computed as
  \[
    b_{2i} = \frac{(n+i)!}{i!} \int_\Delta c(x)^i \diff x \text{.}
  \]
\end{theorem}

Note that $c(\cdot)^i \colon M_\R  \to H^{2i}(B,\R)$ is a vector valued map whose components (after choosing suitable coordinates) are given by homogeneous polynomials of degree $i$.
Thus the integral in Theorem~\ref{BKK1} exists.
Furthermore, although we show that Theorem~\ref{BKK} implies Theorem~\ref{BKK1}, in fact they are equivalent.

\begin{proof}
  Since $H^*(B, \R)$ satisfies Poincar\'{e} duality, it suffices to check that for any $\gamma \in H^{k-2i}(B, \R)$, we have
  \[
    i! \cdot \gamma \cdot b_{2i} = (n+i)! \cdot \gamma \cdot \int_\Delta c(x)^i \diff x,
  \]
  Recall from Theorem~\ref{thm:Leray-Hirsch}, that there are $u_1, \ldots, u_r \in H^*(E_X, \R)$ such that the restrictions $i^*(u_i)$ form a basis for the cohomology of each fibre $X$ where $i \colon X \to E_X$ is the inclusion.
  In particular, there are $y_i \in \R u_1 \oplus \ldots \oplus \R u_r$ such that $i^*(y_i) = x_i$.
  Since $x_{2n}$ is Poincar\'{e} dual to a point (say to a torus fixed point $x \in X$), it follows that $y_{2n} = E \times_T \{x\}$, i.e. $y_{2n}=[S]^*$ is the class Poincar\'e dual to a section of $p$.
  
  Let $[\pt]^* \in H^k(B, \R)$ be the class dual to a point in $B$.
  For $\gamma \in H^{k-2i}(B, \R)$, let $\gamma \cdot b_{2i} = a \cdot [\pt]^*$ for some real number $a \in \R$.
  Then,
  \begin{align*}
    p^*(\gamma) \cdot \rho(\Delta)^{n+i} &= p^*(\gamma) \cdot (p^*(b_{2n+2i}) \cdot y_0 + p^*(b_{2n+2i-2}) \cdot y_2+ \ldots + p^*(b_{2i+2}) \cdot y_{2n-2}+p^*(b_{2i}) \cdot y_{2n})\\
    &= p^*(\gamma \cdot b_{2i}) \cdot y_{2n} = a \cdot p^*([\pt]^*) \cdot E\times_T \{x\} =a \cdot [X]^* \cdot [S]^*= a \in H^{2n+k}(E_X, \R) \simeq \R \text{,}
  \end{align*}
  where $[X]^*$ is the class dual to a fibre of $p$.
  Hence, if we interpret $\gamma \cdot b_{2i}$ and $p^*(\gamma) \cdot \rho(\Delta)^{n+i}$ as real numbers (multiples of classes dual to a point), we get
  \[
    \gamma \cdot b_{2i} = p^*(\gamma) \cdot \rho(\Delta)^{n+i} = F_\gamma(\Delta) \text{.}
  \]
  On the other hand,
  \[
    \gamma \cdot \int_\Delta c(x)^i \diff x = \int_\Delta \gamma \cdot c(x)^i \diff x = I_\gamma(\Delta).
  \]
  The statement follows by Theorem~\ref{BKK}.
\end{proof}

Like Theorem~\ref{BKK}, Theorem~\ref{BKK1} admits a polarized version.

%%%%%%%%%%%%%%%%%%%%%%%%%%%%%%%%%%%%%%%%%%%%%%%%%%%%%%%%%%%%%%%%%

\section{Convex chains}
\label{sec:convex-chains}

In this subsection, we recall how to extend the definition of $I_\gamma:\Pm_\Sigma^+\to \R$ from to all of $\Pm_\Sigma$.
In exposition we mostly follow \cite{PK92}.
Suppose that $\Sigma \subseteq N_\R$ is a \emph{simplicial} fan, i.e., any cone in $\Sigma$ is spanned by part of a basis of $N_\R$.
Then the primitive generators $e_1, \ldots, e_s$ of the rays $\Sigma(1) = \{ \rho_i=\R_{\ge0}e_i \colon i = 1, \ldots, s \}$ naturally induce coordinates on $\Pm_\Sigma$:
a piecewise linear function $h \in \Pm_\Sigma$ is uniquely determined by the tuple $( h(e_1), \ldots, h(e_s))$.
We denote the corresponding coordinates on $\Pm_\Sigma \simeq \R^s$ by $(h_1,\ldots,h_s)$.

We want to work with the space of virtual polytopes $\Pm$.
This is an infinite-dimensional vector space obtained as inverse limit $\Pm =\varinjlim \Pm_\Sigma$ where the limit is taken over all complete fans $\Sigma$ (partially ordered with respect to refinement of fans).
As mentioned above, we consider polynomials, like $I_\gamma$, on this space.
Let us recall the following notion of polynomials on a (possibly infinite-dimensional) real vector space $V$:

On the vector space of functions $g \colon V \to \R$, introduce the \emph{(forward) difference operator}:
\[
  D_v g(x) \coloneqq g(x+v) - g(x) \qquad \text{where $v \in V$.}
\]
Recall that a real topological vector space $U$ is a real vector space equipped with a topology that makes the vector space operations $+ \colon U \times U \to U$ and $\cdot \colon \R \times U \to U$ continuous.
It is a known fact that every finite-dimensional Hausdorff topological vector space has the usual topology.
From now on, we consider all finite-dimensional subspaces $U \subseteq V$ to be equipped with a topology such that $U$ is a Hausdorff topological vector space.

\begin{definition}
  \label{def:poly}
  A function $f \colon V \to \R$ is called a \emph{polynomial of degree $\le m$} if for any $v_1, \ldots, v_{m+1} \in V$
  \[
    D_{v_1} \cdots D_{v_{m+1}} f(x) = 0 \qquad \text{for any $x \in V$,}
  \]
  and the restriction $f|_U \colon U \to \R$ is continuous for any finite-dimensional vector space $U \subseteq V$.
\end{definition}

\begin{remark}
  \label{rem:poly}
  It is straightforward to show that, if $f \colon V \to \R$ is a polynomial and $v_1 \ldots, v_n \in V$ are linearly independent, then for any $\lambda_1, \ldots, \lambda_n \in \R$, we have
  \begin{equation}
    \label{eq:poly}
    f(\lambda_1 v_1 + \ldots + \lambda_n v_n) = \sum_{\alpha \in \Z_{\ge0}^n} D^\alpha f(0) \binom{\lambda_1}{\alpha_1} \cdots \binom{\lambda_n}{\alpha_n}
  \end{equation}
  where $D^\alpha \coloneqq D_{v_1}^{\alpha_1} \cdots D_{v_n}^{\alpha_n}$ and $\binom{\lambda}{\alpha} = \tfrac{\lambda \cdot (\lambda-1) \cdots (\lambda-\alpha+1)}{\alpha!}$ denotes the binomial coefficient.
  Note that Definition~\ref{def:poly} deviates from \cite[Section 2, Definition 1]{PK92} by the additional assumption of continuity.
  Indeed, without this assumption, equation~(\ref{eq:poly}) would be true for $\lambda_1, \ldots, \lambda_n \in \Q$ (which is sufficient in \cite{PK92}).
  However, in our setting it is more natural to work over the field of real numbers, which makes this further assumption necessary.
\end{remark}

In view of Remark~\ref{rem:poly}, we may rephrase Definition~\ref{def:poly} as follows:

\begin{proposition}
  A function $f \colon V \to \R$ is a polynomial of degree $\le m$ if and only if its restriction to any finite dimensional subspace $U \subseteq V$ is a polynomial of degree $\le m$ in the usual sense, i.e., an element in $\Sym(U^*)$.
\end{proposition}

Recall the objective of this section:
Any smooth function $f\colon M_\R \to \R$ defines a function $I_f$ on the cone of convex polytopes via (here $\mu$ denotes the standard Lebesgue measure on $M_\R$)
\[
  I_f \colon \Pm^+ \to \R; I_f(\Delta) =\int_\Delta f(x) \diff \mu \text{.}
\]
Our goal is to extend $I_f$ to the whole space of virtual polytopes $\Pm$.
We recall how this is done following \cite{PK92}.

Observe that $I_f \colon \Pm^+ \to \R$ is a \emph{finitely additive measure} (also called a \emph{valuation}) on the cone of convex polytopes, i.e. for any two convex polytopes $P,Q \in \Pm^+$ such that $P \cup Q$ is in $\Pm^+$, we have
\[
  I_f(P \cup Q) = I_f(P) + I_f(Q) - I_f(P\cap Q) \text{.}
\]
A fundamental idea in \cite{PK92} is the introduction of the group of \emph{convex chains} $Z(M_\R)$, i.e. the additive group consisting of functions of the form $\alpha = \sum_{i=1}^k n_i  \mathds{1}_{P_i}$ for integers $n_i$ and convex polytopes $P_i \in \Pm^+$ where $\mathds{1}_{P_i}$ denotes the characteristic function of the set $P_i$.
The \emph{degree} of a convex chain $\alpha \in Z(M_\R)$ is defined to be $\sum_{i=1}^k n_i$ if $\alpha = \sum_{i=1}^kn_i\mathds{1}_{P_i}$ for some $P_i \in \Pm^+$ (this is a well-defined number by \cite[Proposition-Definition 2.1]{PK92}).
In \cite[Proposition-Definition 2.3]{PK92}, Pukhlikov and the second author observe that the Minkowski sum of polytopes induces a ring structure on $Z(M_\R)$ via
\[
  \mathds{1}_P \star \mathds{1}_Q = \mathds{1}_{P+Q} \qquad \text{for $P, Q \in \Pm^+$.}
\]
Clearly, the semigroup $\Pm^+$ can be embedded in the algebra of convex chains $Z(M_\R)$.
Furthermore, in \cite[Section 6]{PK92} it is shown that $\Pm$ can be identified with the set of invertible convex chains of degree $1$.

Observe that finitely additive measures (as defined above) are in correspondence with homomorphisms of additive groups $\varphi \colon Z(M_\R) \to \R$.
A finitely additive measure is \emph{polynomial of degree $\le m$} if for each $\alpha \in Z(M_\R)$ the function $M_\R \to \R$ given by $\lambda \mapsto \varphi(\tau(\lambda, \alpha))$ is polynomial of degree $\le m$ where $\tau(\lambda, \cdot) \colon Z(M_\R) \to Z(M_\R)$ denotes the linear map induced by translation by $\lambda$, i.e. $\tau(\lambda, \alpha) (x) = \alpha(x - \lambda)$.
A fundamental result about polynomial finitely additive measures is the following theorem:

\begin{theorem}[{\cite[Corollary 7.5]{PK92}}]
  \label{thm:pol-measure}
  A polynomial finitely additive measure $\varphi \colon Z(M_\R)\to \R$ of degree $\le k$ restricted to the space of virtual polytopes $\varphi \colon \Pm \to \R$ is a polynomial of degree $\le \dim(M_\R) + k$ in the sense of Definition~\ref{def:poly} (however the continuity assumption might not be satisfied in general).
\end{theorem}

We now apply the results of \cite{PK92} summarized above to our situation.
Note that $I_f \colon \Pm^+\to \R$ can be naturally extended to a homomorphism of additive groups $Z(M_\R) \to \R$ via
\[
  I_f(\alpha) = \int_{M_\R} f(x) \cdot \alpha(x) \diff\mu = \sum_{i=1}^k n_i \int_{P_i} f(x) \diff \mu \qquad \text{for $\alpha = \sum_{i=1}^k n_i \mathds{1}_{P_i} \in Z(M_\R)$.}
\]
Furthermore, $I_f \colon Z(M_\R) \to \R$ is a polynomial finitely additive measure of degree $\le \deg(f)$ which follows by the observation that the following map is a polynomial of degree $\le \deg(f)$ for any $\alpha = \sum_{i=1}^kn_i\mathds{1}_{P_i} \in Z(M_\R)$:
\[
  I_f(\tau(\cdot,\alpha)) \colon M_\R \to \R; \quad \lambda \mapsto \int_{\tau(\lambda, \alpha)} f(x) \diff \mu = \sum_{i=1}^k n_i \cdot \int_{P_i+\lambda} f(x) \diff \mu = \sum_{i=1}^k n_i \cdot \int_{P_i} f(x-\lambda) \diff \mu \text{.}
\]
By Theorem~\ref{thm:pol-measure}, the restriction of $I_f$ to the space of virtual polytopes is a polynomial of degree $\le \dim(M_\R) + \deg(f)$.
The continuity assumption of Definition~\ref{def:poly} can be seen as follows.

By \cite[Corollary 2.2]{PK92}, the natural inclusion $\Pm^+ \to Z(M_\R); P \mapsto \mathds{1}_P$ extends to an isomorphism of groups from $\Pm$ (with Minkowski addition) onto the invertible elements $\Pm^*$ of $Z(M_\R)$ of degree $1$ (equipped with the multiplication operation).
The scalar multiplication on $\Pm$ naturally carries over to $\Pm^*$:
\[
  \left( \mathds{1}_P\right)^\lambda = \begin{cases}
    \mathds{1}_{\lambda P} & \text{if $\lambda\ge0$,}\\
    \mathds{1}^{-1}_{-\lambda P} & \text{if $\lambda<0$}
  \end{cases} \qquad \text{for any convex polytope $P \subseteq M_\R$ and every real scalar $\lambda$.}
\]
We obtain an isomorphism of vector spaces between $\Pm \simeq \Pm^*$.
By our choice of the topology from above, restricting this isomorphism to corresponding finite-dimensional subspaces yields a homeomorphism.
Let $U \subseteq \Pm^*$ be a finite-dimensional subspace and choose a basis $\alpha_1, \ldots, \alpha_n \in U$.
Then two elements $\lambda = \alpha_1^{\lambda_1} \star \cdots \star \alpha_n^{\lambda_n}$ and $\mu = \alpha_1^{\mu_1} \star \cdots \star \alpha_n^{\mu_n}$ in $U$ are close to each other if and only if the scalars $\lambda_i$ and $\mu_i$ are close to each other for any index $i$.
It is straightforward to show that $I_f(\lambda)$ and $I_f( \mu )$ are close to each other, i.e. $I_f$ is continuous.

Recall that a polynomial $f \colon V \to \R$ on a (possibly infinite-dimensional) vector space $V$ is called \emph{homogeneous of degree $m$} if for all $v \in V$ and every $t \in \R$, we have $f(t\cdot v) = t^m \cdot f(v)$.

To summarize the above, let us state the following theorem.
\begin{theorem}
  \label{ultimatepoly}
  If $f\colon M_\R \to\R$ is a homogeneous polynomial of degree $m$, then the  function $I_f\colon \Pm^+\to \R; \Delta \mapsto I_f(\Delta) =\int_\Delta f(x) \diff \mu$ admits a unique extension to a homogeneous polynomial of degree $n+m$ on $\Pm$.
\end{theorem}

%%%%%%%%%%%%%%%%%%%%%%%%%%%%%%%%%%%%%%%%%%%%%%%%%%%%%%%%%%%%%%%%%

\section{Proof of the \BKK theorems}
\label{BKKproof}

In this section we prove Theorem~\ref{BKK}.
For the reader's convenience, we recall the used notation.
Let $p \colon E \to B$ be a principal torus bundle with respect to an $n$-dimensional torus $T$ over a smooth compact orientable manifold $B$ of real dimension $k$.
The character lattice of $T$ we denote by $M$.
Let $\Sigma \subseteq N_\R = \Hom_\Z(M,\R)$ be a smooth complete fan with rays $\rho_1, \ldots, \rho_s$.
The primitive ray generators we denote by $e_1, \ldots, e_s$.
Let $h_1, \ldots, h_s$ be the basis of $\Pm_\Sigma$ consisting of piecewise linear functions which vanish on every ray of $\Sigma$ except one where it evaluates with $1$ on the corresponding primitive ray generator.
Let $X = X_\Sigma$ be the toric variety corresponding to $\Sigma$ and denote the total space of the associated toric bundle by $E_X$.
To keep notation simple we use the same notation $p \colon E_X \to B$ for the projection map of the toric bundle.
Fix $i \le \frac{k}{2}$ and a class $\gamma \in H^{k-2i}(B,\R)$.
The rays $\rho_i$ correspond to divisors in $X$ which give rise to divisors $D_i$ in $E_X$.
For a virtual polytope $\Delta \in \Pm_\Sigma$, we introduced
\[
  \rho(\Delta) = \sum_{i=1}^s h_\Delta(e_i) [D_i] \in H^2(E_X, \R) \text{,}
\]
where $[D_i]$ is the class dual to $D_i \subseteq E_X$ and $h_\Delta \colon N_\R \to \R$ is the support function of $\Delta$.
Further, we introduce
\[
  F_\gamma \colon \Pm_\Sigma \to \R; F_\gamma(\Delta) = \rho(\Delta)^{n+i} \cdot p^*(\gamma) \text{.}
\]
We noticed that $F_\gamma$ is a homogeneous polynomial of degree $n+i$ on $\Pm_\Sigma$.

Recall that for any character $\lambda \in M$ we have an associated complex line bundle $\mathcal{L}_\lambda$ on $B$.
Taking Chern classes and extending by linearity, we obtain a morphism of vector spaces $c \colon M_\R \to H^2(B,\R)$.
Let $f_\gamma \colon M_\R \to \R$ be the function $f_\gamma(x) = c(x)^i \cdot \gamma$.
We defined a map $I_\gamma \colon \Pm_\Sigma\to \R$ which is explicitly given for $\Delta \in \Pm_\Sigma^+$ by
\[
  I_\gamma(\Delta) = \int_\Delta f_\gamma(x) \diff\mu\text{.}
\]
In the previous section, we saw how to extend this definition to all of $\Pm_\Sigma$.
Furthermore, by Theorem~\ref{ultimatepoly},  $I_\gamma$ is a homogeneous polynomial of degree $n+i$ on $\Pm_\Sigma$.

We prove Theorem~\ref{BKK} by induction on the parameter $0 \le i \le \frac{k}{2}$.

Let us start with the base case $i=0$.
If $i=0$, then $\gamma \in H^k(B, \R)$ is a multiple of the class dual to a point.
For simplicity, let us assume that this multiple is $1$.
So $p^*(\gamma)$ is the class dual to a fibre, which is a toric variety, and thus $\rho(\Delta)^n\cdot p^*(\gamma)$ coincides with the degree of the divisor in $X$ corresponding to $\Delta$.
By the classical \BKK theorem, this can be computed by the normalized volume of $\Delta$ which equals to $n! \cdot I_\gamma(\Delta)$.
In other words, for $i=0$, Theorem~\ref{BKK} reduces to the classical \BKK theorem.

As induction hypothesis, suppose that we know Theorem~\ref{BKK} for some $i-1\ge0$.
The induction step consists of proving that Theorem~\ref{BKK} is also true for $i$.
Since both $F_\gamma$ and $I_\gamma$ are homogeneous polynomials of the same degree $n+i$, in order to show equality between $(n+i)!\cdot I_\gamma(\Delta)$ and $i!\cdot F_\gamma(\Delta)$, it suffices to show that all their partial derivatives of order $n$ coincide.
In other words, it suffices to consider differential monomials $\partial_{i_1}^{k_1}\dots\partial_{i_r}^{k_r}$ where  $\partial_i= \partial/\partial_{h_i}$ are the partial derivatives along the coordinate vectors of $\Pm_\Sigma \simeq \R^s$ and $\sum_{i=1}^r k_i =n$.
Let us call the number $\sum_{i=1}^r (k_i-1)$ the \emph{multiplicity of the monomial} $\partial_{i_1}^{k_1}\dots\partial_{i_r}^{k_r}$.
In particular, a monomial has multiplicity $0$ if and only if it is square free.

The proof of equality of the partial derivatives of order $n$ of $(n+i)! \cdot I_\gamma(\Delta)$ and $i! \cdot F_\gamma(\Delta)$ is by induction on the multiplicity $m$ of the applied differential monomial.
We will refer to the induction over $i$ as the ``outer induction'' and we are going to call the induction over $m$ the ``inner induction''.

The base case of the inner induction (i.e., the case of square free differential monomials) is treated in the next two subsections.
The result of these calculations is summarized in Proposition~\ref{sqfree}.

%%%%%%%%%%%%%%%%%%%%%%%%%%%%%%%%%%%%%%%%%%%%%%%%%%%%%%%%%%%%%%%%%

\subsection{\texorpdfstring{Differentiation of $I_\gamma$}{Differentiation of I-gamma}}
The main result of this subsection is Lemma~\ref{Ider} below which computes the partial derivative of $I_\gamma$ with respect to square free monomial.
Here, we work with a slightly more general situation than needed for the proof of Theorem~\ref{BKK}: we assume that $\Sigma$ is a simplicial fan, and we define the function $I_f\colon \Pm \to \R$ via $I_f(\delta)= \int_\delta f \diff\mu$ for any smooth function $f$ on $M_\R$.
In the proof of Theorem~\ref{BKK}, we use $f=f_\gamma$.

\begin{lemma}
  \label{Ider}
  Let $I = \{ i_1, \ldots, i_r\} \subseteq \{ 1, \ldots, s \}$ be a subset and $k_1, \ldots, k_r$ positive integers.
  If $\Delta$ is a polytope in the interior of $\Pm_\Sigma^+$ and $\rho_{i_1},\ldots,\rho_{i_r}$ do not span a cone in $\Sigma$, then we have
  \[
    \partial_{i_1}^{k_1} \cdots \partial_{i_r}^{k_r} \left(I_f|_{\Pm_\Sigma}\right)(\Delta) = 0 \text{.}
  \]
  However, if $r = n$ and $\rho_{i_1}, \ldots, \rho_{i_n}$ span a cone in $\Sigma$ dual to the vertex $A\in \Delta$, we have 
  \[
    \partial_I \left(I_f|_{\Pm_\Sigma}\right)(\Delta) = f(A) \cdot |\det (e_{i_1}, \ldots, e_{i_n})| \text{.}
  \]
\end{lemma}
Here $\det(e_{i_1}, \ldots, e_{i_n})$ denotes the determinant of the matrix whose $j$-th column is given by the vector $e_{i_j}$.
Furthermore, $\partial_I$ denotes the partial derivative $\partial / \partial h_{i_1} \cdots \partial / \partial h_{i_n}$ along the coordinate vectors.

The proof of Lemma~\ref{Ider} relies on a folklore result on convex chains.
We continue to use the notation from above.
Fix real numbers $\{ \lambda_i \}_{i \in I}$ and define the virtual polytope $h + \sum_{i \in I}\lambda_i h_i$ where $h$ denotes the support function of the polytope $\Delta$.
If $\{ \lambda_i \}_{i \in I}$ are sufficiently small, the virtual polytope $h + \sum_{i \in I} \lambda_i h_i$ is in the interior of $\Pm^+_\Sigma$ in which case it corresponds to the polytope obtained from $\Delta$ by moving the facets corresponding to the rays $\rho_{i_j}$ according to the coefficients $\lambda_{i_j}$.
For any subset $J \subseteq I$ we introduce the virtual polytope $\Delta(J)\coloneqq h + \sum_{j \in J} \lambda_j h_j$.
Furthermore, we  define the convex chain
\[
  \sigma(\lambda_{i_1}, \ldots, \lambda_{i_r}) = \sum_{J \subseteq I} (-1)^{r+|I_-|+|J|} \cdot \mathds{1}_{\Delta(J)} \qquad \text{where $I_- \coloneqq \{ i \in I \colon \lambda_i<0\}$.}
\]
The following folklore result will be needed in the proof of Lemma~\ref{Ider}.
\begin{proposition}
  \label{prop:cc}

  Let $I = \{i_1, \ldots, i_r\} \subseteq \{ 1, \ldots, s\}$ be a subset and let $\{ \lambda_i\}_{i \in I}$ be sufficiently small real numbers.
  If any  $\lambda_i = 0$ for $i \in I$ or $\rho_{i_1}, \ldots, \rho_{i_r}$ do not span a cone in $\Sigma$, then $\sigma(\lambda_{i_1}, \ldots, \lambda_{i_r}) = 0$.  
  If all $\lambda_i \neq 0$ for $i \in I$, and the cardinality of $I$ is $n = \dim(N_\R)$, and $\rho_{i_1}, \ldots, \rho_{i_n}$ span a cone in $\Sigma$ dual to the vertex $A \in \Delta$, then
  \[
     \sigma(\lambda_{i_1}, \ldots, \lambda_{i_n}) =  \mathds{1}_{A + \Pi(\lambda_{i_1}, \ldots, \lambda_{i_n})} \text{.}
  \]
\end{proposition}
Here, $A+\Pi(\lambda_{i_1}, \ldots, \lambda_{i_n})$ denotes the half-open parallelepiped spanned by the vectors $\lambda_{i_1} e_{i_1}, \ldots, \lambda_{i_n} e_{i_n}$ shifted by the vertex $A$:
\[
  \Pi\left( \lambda_{i_1}, \ldots, \lambda_{i_n} \right) = \left\{ \sum_{i \in I} a_i \lambda_i e_i \colon 0 < a_i \le 1 \; \text{if} \; \lambda_i >0 \; \text{and} \; 0 \le a_i <1 \; \text{if} \; \lambda_i <0 \right\} \text{.}
\] 
See Figures~\ref{fig1} and~\ref{fig2} for an illustration of Lemma~\ref{prop:cc} in dimension $2$.
\begin{figure}[ht!]
\centering
\begin{tikzcd}[column sep=tiny, row sep=tiny]
\parbox{2.25cm}{\begin{tikzpicture}[scale=.5]
  \draw[thick] (1,-.5) -- (2,-.75) -- (2,2.5) -- (1,3);
  \draw[thick] (2,2.5) -- (2,3.5) -- (-.6,4.8) -- (-1,4);
  \draw[thick] (1,3) -- (1,4);
  \fill[fill=gray!25] (-1,0) -- (-2,2) -- (-1,4) -- (1,3) -- (1,-.5) -- cycle;
  \draw[thick] (-1,0) -- (-2,2) -- (-1,4) -- (1,3) -- (1,-.5) -- cycle;
  \fill[fill=gray,draw=black,thin] (-1,0) circle (5pt);
  \fill[fill=gray,draw=black,thin] (-2,2) circle (5pt);
  \fill[fill=gray,draw=black,thin] (-1,4) circle (5pt);
  \fill[fill=gray,draw=black,thin] (1,3) circle (5pt);
  \fill[fill=gray,draw=black,thin] (1,-.5) circle (5pt);
  \fill (2,2.5) circle (4pt);
  \fill (2,-.75) circle (4pt);
  \fill (1,4) circle (4pt);
  \fill (-.6,4.8) circle (4pt);
  \fill (2,3.5) circle (4pt);
\end{tikzpicture}} & - &
\parbox{2.25cm}{\begin{tikzpicture}[scale=.5]
  \draw[thick] (1,-.5) -- (2,-.75) -- (2,2.5) -- (1,3);
  \draw[draw=none,thick] (2,2.5) -- (2,3.5) -- (-.6,4.8) -- (-1,4);
  \draw[draw=none,thick] (1,3) -- (1,4);
  \fill[fill=gray!25] (-1,0) -- (-2,2) -- (-1,4) -- (1,3) -- (1,-.5) -- cycle;
  \draw[thick] (-1,0) -- (-2,2) -- (-1,4) -- (1,3) -- (1,-.5) -- cycle;
  \fill[fill=gray,draw=black,thin] (-1,0) circle (5pt);
  \fill[fill=gray,draw=black,thin] (-2,2) circle (5pt);
  \fill[fill=gray,draw=black,thin] (-1,4) circle (5pt);
  \fill[fill=gray,draw=black,thin] (1,3) circle (5pt);
  \fill[fill=gray,draw=black,thin] (1,-.5) circle (5pt);
  \fill (2,2.5) circle (4pt);
  \fill (2,-.75) circle (4pt);
  \fill[fill=none,draw=none] (1,4) circle (4pt);
  \fill[fill=none,draw=none] (-.6,4.8) circle (4pt);
  \fill[fill=none,draw=none] (2,3.5) circle (4pt);
\end{tikzpicture}} & - &
\parbox{2cm}{\begin{tikzpicture}[scale=.5]
  \draw[thick,draw=none] (1,-.5) -- (2,-.75) -- (2,2.5) -- (1,3);
  \draw[thick,draw=none] (2,2.5) -- (2,3.5) -- (-.6,4.8) -- (-1,4);
  \draw[thick] (1,3) -- (1,4) -- (-.6,4.8) -- (-1,4);
  \fill[fill=gray!25] (-1,0) -- (-2,2) -- (-1,4) -- (1,3) -- (1,-.5) -- cycle;
  \draw[thick] (-1,0) -- (-2,2) -- (-1,4) -- (1,3) -- (1,-.5) -- cycle;
  \fill[fill=gray,draw=black,thin] (-1,0) circle (5pt);
  \fill[fill=gray,draw=black,thin] (-2,2) circle (5pt);
  \fill[fill=gray,draw=black,thin] (-1,4) circle (5pt);
  \fill[fill=gray,draw=black,thin] (1,3) circle (5pt);
  \fill[fill=gray,draw=black,thin] (1,-.5) circle (5pt);
  \fill[fill=none,draw=none] (2,2.5) circle (4pt);
  \fill[fill=none,draw=none] (2,-.75) circle (4pt);
  \fill (1,4) circle (4pt);
  \fill (-.6,4.8) circle (4pt);
  \fill[fill=none,draw=none] (2,3.5) circle (4pt);
\end{tikzpicture}} & + &
\parbox{2cm}{\begin{tikzpicture}[scale=.5]
  \draw[thick,draw=none] (1,-.5) -- (2,-.75) -- (2,2.5) -- (1,3);
  \draw[thick,draw=none] (2,2.5) -- (2,3.5) -- (-.6,4.8) -- (-1,4);
  \draw[thick,draw=none] (1,3) -- (1,4);
  \fill[fill=gray!25] (-1,0) -- (-2,2) -- (-1,4) -- (1,3) -- (1,-.5) -- cycle;
  \draw[thick] (-1,0) -- (-2,2) -- (-1,4) -- (1,3) -- (1,-.5) -- cycle;
  \fill[fill=gray,draw=black,thin] (-1,0) circle (5pt);
  \fill[fill=gray,draw=black,thin] (-2,2) circle (5pt);
  \fill[fill=gray,draw=black,thin] (-1,4) circle (5pt);
  \fill[fill=gray,draw=black,thin] (1,3) circle (5pt);
  \fill[fill=gray,draw=black,thin] (1,-.5) circle (5pt);
  \fill[fill=none,draw=none] (2,2.5) circle (4pt);
  \fill[fill=none,draw=none] (2,-.75) circle (4pt);
  \fill[fill=none,draw=none] (1,4) circle (4pt);
  \fill[fill=none,draw=none] (-.6,4.8) circle (4pt);
  \fill[fill=none,draw=none] (2,3.5) circle (4pt);
\end{tikzpicture}} & = &
\parbox{.75cm}{\begin{tikzpicture}[scale=.5]
  \draw[thick,dashed] (0,3) -- (0,4);
  \draw[thick] (0,4) -- (1,3.5);
  \draw[thick] (1,3.5) -- (1,2.5);
  \draw[thick,dashed] (0,3) -- (1,2.5);
  \fill (1,2.5) circle (4pt);
  \fill (1,3.5) circle (4pt);
  \fill (0,4) circle (4pt);
  \fill[fill=gray,draw=black,thin] (0,3) circle (5pt);
  \fill[fill=none,draw=none] (0,4.8) circle (4pt);
  \fill[fill=none,draw=none] (0,-.75) circle (4pt);
\end{tikzpicture}}\\
\Delta + \lambda_1 x_1 + \lambda_2 x_2 && \Delta + \lambda_1 x_1 && \Delta + \lambda_2 x_2 && \Delta
\end{tikzcd}
\caption{Alternating sum of characteristic functions for two adjacent edges.}
\label{fig1}
\begin{tikzcd}[column sep=tiny, row sep=tiny]
\parbox{2.25cm}{\begin{tikzpicture}[scale=.5]
  \draw[thick] (1,-.5) -- (2,-.75) -- (2,2.5) -- (1,3);
  \draw[thick] (-2,2) -- (-2.5,3) -- (-1.8,4.4) -- (-1,4);
  \fill[fill=gray!25] (-1,0) -- (-2,2) -- (-1,4) -- (1,3) -- (1,-.5) -- cycle;
  \draw[thick] (-1,0) -- (-2,2) -- (-1,4) -- (1,3) -- (1,-.5) -- cycle;
  \fill[fill=gray,draw=black,thin] (-1,0) circle (5pt);
  \fill[fill=gray,draw=black,thin] (-2,2) circle (5pt);
  \fill[fill=gray,draw=black,thin] (-1,4) circle (5pt);
  \fill[fill=gray,draw=black,thin] (1,3) circle (5pt);
  \fill[fill=gray,draw=black,thin] (1,-.5) circle (5pt);
  \fill (2,2.5) circle (4pt);
  \fill (2,-.75) circle (4pt);
  \fill (-2.5,3) circle (4pt);
  \fill (-1.8,4.4) circle (4pt);
\end{tikzpicture}}& - &
\parbox{2cm}{\begin{tikzpicture}[scale=.5]
  \draw[thick] (1,-.5) -- (2,-.75) -- (2,2.5) -- (1,3);
  \fill[fill=gray!25] (-1,0) -- (-2,2) -- (-1,4) -- (1,3) -- (1,-.5) -- cycle;
  \draw[thick] (-1,0) -- (-2,2) -- (-1,4) -- (1,3) -- (1,-.5) -- cycle;
  \fill[fill=gray,draw=black,thin] (-1,0) circle (5pt);
  \fill[fill=gray,draw=black,thin] (-2,2) circle (5pt);
  \fill[fill=gray,draw=black,thin] (-1,4) circle (5pt);
  \fill[fill=gray,draw=black,thin] (1,3) circle (5pt);
  \fill[fill=gray,draw=black,thin] (1,-.5) circle (5pt);
  \fill (2,2.5) circle (4pt);
  \fill (2,-.75) circle (4pt);
  \fill[fill=none,draw=none] (0,4.4) circle (4pt);
\end{tikzpicture}}& - &
\parbox{1.75cm}{\begin{tikzpicture}[scale=.5]
  \draw[thick] (-2,2) -- (-2.5,3) -- (-1.8,4.4) -- (-1,4);
  \fill[fill=gray!25] (-1,0) -- (-2,2) -- (-1,4) -- (1,3) -- (1,-.5) -- cycle;
  \draw[thick] (-1,0) -- (-2,2) -- (-1,4) -- (1,3) -- (1,-.5) -- cycle;
  \fill[fill=gray,draw=black,thin] (-1,0) circle (5pt);
  \fill[fill=gray,draw=black,thin] (-2,2) circle (5pt);
  \fill[fill=gray,draw=black,thin] (-1,4) circle (5pt);
  \fill[fill=gray,draw=black,thin] (1,3) circle (5pt);
  \fill[fill=gray,draw=black,thin] (1,-.5) circle (5pt);
  \fill[fill=none,draw=none] (0,-.75) circle (4pt);
  \fill (-2.5,3) circle (4pt);
  \fill (-1.8,4.4) circle (4pt);
\end{tikzpicture}}& + &
\parbox{1.55cm}{\begin{tikzpicture}[scale=.5]
  \fill[fill=gray!25] (-1,0) -- (-2,2) -- (-1,4) -- (1,3) -- (1,-.5) -- cycle;
  \draw[thick] (-1,0) -- (-2,2) -- (-1,4) -- (1,3) -- (1,-.5) -- cycle;
  \fill[fill=gray,draw=black,thin] (-1,0) circle (5pt);
  \fill[fill=gray,draw=black,thin] (-2,2) circle (5pt);
  \fill[fill=gray,draw=black,thin] (-1,4) circle (5pt);
  \fill[fill=gray,draw=black,thin] (1,3) circle (5pt);
  \fill[fill=gray,draw=black,thin] (1,-.5) circle (5pt);
  \fill[fill=none,draw=none] (0,-.75) circle (4pt);
  \fill[fill=none,draw=none] (0,4.4) circle (4pt);
\end{tikzpicture}}& = & 0\\
\Delta + \lambda_1 x_1 + \lambda_3 x_3 && \Delta + \lambda_1 x_1 && \Delta + \lambda_3 x_3 && \Delta
\end{tikzcd}
\caption{Alternating sum of characteristic functions for two disjoint edges.}
\label{fig2}
\end{figure}

\begin{proof}[Sketch of proof]
  For the reader's convenience we include a sketch of proof.

  It is straightforward to check that $\sigma(\lambda_{i_1}, \ldots, \lambda_{i_r})$ vanishes if any $\lambda_i=0$ for $i \in I$.
  Hence suppose no $\lambda_i$ vanishes for $i \in I$.
  It turns out to be more convenient to work with the case that all $\lambda_i>0$ for $i \in I$.
  We can always reduce to this case as follows.
  Let $I_- \coloneqq \{ i \in I \colon \lambda_i<0\}$ and set $\Delta' = h + \sum_{i \in I_-} \lambda_i h_i$.
  Let $h'$ be the support function of $\Delta'$ and define $\Delta'(J) = h' + \sum_{j \in J} |\lambda_j| h_j$ for any subset $J \subseteq I$.
  Consider the convex chain
  \[
    \sigma'(|\lambda_{i_1}|, \ldots, |\lambda_{i_r}|) = \sum_{J \subseteq I} (-1)^{r+|J|} \mathds{1}_{\Delta'(J)}\text{.}
  \]
  It is straightforward to verify that $\sigma(\lambda_{i_1}, \ldots, \lambda_{i_r}) = (-1)^{|I_-|}  \cdot \sigma'(|\lambda_{i_1}|, \ldots, |\lambda_{i_r}|)$, and thus we may assume from now on that $\lambda_i>0$ for every $i \in I$.

  Suppose that the facet $F_i$ of $\Delta$ corresponding to the ray $\rho_i \in \Sigma$ is given by the supporting hyperplane $H_i = \{ x \in M_\R \colon \langle x, e_i \rangle = c_i\}$ for some $c_i \in \R$.
  Note that the polytope $\Delta(I)$ is obtained from $\Delta$ by moving the supporting hyperplanes $H_i$ for $i \in I$ outwards.
  We say that a point $x \in \Delta(I)$ is \emph{beyond the hyperplane $H_i$} for $i \in I$ if $\langle x, e_i \rangle >c_i$.
  We get a map from the points of $\Delta(I)$ to the powerset $2^I$ of $I$:
  \[
    \Delta(I) \to 2^I; x \mapsto \mathcal{I}(x) \coloneqq \{ i \in I \colon \langle x, e_i \rangle >c_i\} \text{.}
  \]
  For $J \subseteq I$ define the region $\mathcal{C}(J)$ as the set of points $x \in \Delta(I)$ where  $\mathcal{I}(x) = J$ is constant.
  The regions $\mathcal{C}(J)$ form a (disjoint) decomposition of $\Delta(I)$ into (possibly half-open) polytopes.
  Points in these regions behave well with respect to the characteristic functions $\mathds{1}_{\Delta(J)}$ which appear in the convex chain $\sigma(\lambda_{i_1}, \ldots, \lambda_{i_r})$:
  \[
    \mathds{1}_{\Delta(J)} = \begin{cases}
      1 & \text{if $\mathcal{I}(x) \subseteq J$,}\\
      0 & \text{otherwise}
    \end{cases} \qquad \text{for $x \in \Delta(I)$ and $J \subseteq I$.}
  \]
  For $x \in \Delta(I)$, a straightforward computation yields 
  \[
    \sigma(\lambda_{i_1}, \ldots, \lambda_{i_r})(x) = \begin{cases}
      0 & \text{if $|\mathcal{I}(x)|<r$,}\\
      (-1)^r & \text{if $|\mathcal{I}(x)|=r$.}
    \end{cases}
  \]
  Hence, $\sigma(\lambda_{i_1}, \ldots, \lambda_{i_r}) = (-1)^r \cdot \mathds{1}_{\mathcal{C}(I)}$.
  It is straightforward to show that $\mathcal{C}(I)=\emptyset$ if $\rho_{i_1}, \ldots, \rho_{i_r}$ do not span a cone in $\Sigma$.
  If $|I| = n = \dim( N_\R)$ and $\rho_{i_1}, \ldots, \rho_{i_n}$ span a cone in $\Sigma$ dual to the vertex $A \in \Delta$, then a straightforward computation shows that $\mathcal{C}(I) = A + \Pi(\lambda_{i_1}, \ldots, \lambda_{i_n})$.
\end{proof}

\begin{proof}[Proof of Lemma~\ref{Ider}]
  To keep notation simple assume $\rho_{i_j} = \rho_{j}$.
  
  For a monomial $\partial_1^{k_1} \cdots\partial_r^{k_r}$, let $\partial_I = \partial_1 \cdots \partial_r$ be the corresponding square free monomial.
  It is enough to show that $\partial_I F_\gamma(\Delta)  = 0$ in order  to prove that $\partial_1^{k_1} \cdots \partial_r^{k_r} F_\gamma(\Delta)  = 0$.

  For a smooth function $g$ on $\Pm_\Sigma$, the partial derivative $\partial_I g(x)$ can be expressed as
  \begin{equation}
    \label{difdiff}
    \partial_I g(x) = \lim_{\lambda_1, \ldots, \lambda_r\to 0}\rleft( \frac{1}{\lambda_1 \cdots \lambda_r} D_{\lambda_1h_1}\ldots D_{\lambda_rh_r} g(x) \rright) \text{,}
  \end{equation}
  where $D_{\lambda_i h_i}$ denotes the difference operator introduced above and $\lambda_1, \ldots, \lambda_r \in \R$.
  
  Applying equality~(\ref{difdiff}) to the function $I_f|_{\Pm_\Sigma}$ one gets 
  \[
    \partial_I  I_f(\Delta) = \lim_{\lambda_1,\ldots,\lambda_r\to 0} \frac{1}{\lambda_1 \cdots \lambda_r} \int_{M_\R} f(x) \cdot (-1)^{|I_-|} \cdot \sigma(\lambda_1, \ldots, \lambda_r)(x) \diff \mu
  \]
  where $\sigma(\lambda_1, \ldots, \lambda_r)$ is the convex chain defined above and $I_- = \{ i = 1, \ldots, r \colon \lambda_i <0\}$.
  By Proposition~\ref{prop:cc}, the convex chain $\sigma(\lambda_1, \ldots, \lambda_r)$ vanishes if $\rho_1, \ldots, \rho_r$ do not form a cone in $\Sigma$.
  By the same proposition, if $r = n$ and $\rho_1, \ldots, \rho_n$ generate a cone in $\Sigma$ dual to the vertex $A \in \Delta$, it is equal to the characteristic function $\mathds{1}_{A+\Pi(\lambda_1,\ldots,\lambda_n)}$ of the half-open parallelepiped spanned by the vectors $\lambda_1 e_1,\ldots,\lambda_n e_n$ at the point $A$.
  In this latter case, i.e. $\rho_1, \ldots, \rho_n$ span a cone in $\Sigma$, let $T \colon \R^n \to M_\R$ be the diffeomorphism given by $T(x_1, \ldots, x_n) = x_1 e_1 + \ldots + x_n e_n +A$.
  Then by a straightforward calculation, we obtain:
  \[
    \int_{A + \Pi(\lambda_1, \ldots, \lambda_n)} f(x) \diff\mu = (-1)^{|I_-|} \cdot \int_0^{\lambda_1} \cdots \int_0^{\lambda_n} f(T(x_1, \ldots, x_n)) \cdot |\det(e_1, \ldots, e_n)| \diff x_1 \cdots \diff x_n \text{,}
  \]
  where $\det(e_1, \ldots, e_n)$ is the determinant of the Jacobian matrix of $T$.
  So the lemma follows by the fundamental theorem of calculus:
  \[
    \lim_{\lambda_1,\ldots,\lambda_n\to 0} \frac{1}{\lambda_1\ldots\lambda_n} \int_{A+\Pi(\lambda_1 e_1,\ldots,\lambda_n e_n)} (-1)^{|I_-|} \cdot f(x) \diff \mu = f(A) \cdot |\det (e_1,\ldots, e_n)| \text{.}\qedhere
  \]
\end{proof}

%%%%%%%%%%%%%%%%%%%%%%%%%%%%%%%%%%%%%%%%%%%%%%%%%%%%%%%%%%%%%%%%%

\subsection{\texorpdfstring{Differentiation of $F_\gamma$}{Differentiation of F-gamma}}

Next we verify the base case of square free differential monomials for $F_\gamma$:

\begin{lemma}
  \label{Fder}
  Let $I = \{ i_1, \ldots, i_r\} \subseteq \{ 1, \ldots, s \}$ be a subset and $k_1, \ldots, k_r$ positive integers.
  If $\Delta$ is a polytope in the interior of $\Pm_\Sigma^+$ such that $\rho_{i_1},\ldots,\rho_{i_r}$ do not span a cone in $\Sigma$, we have
  \[
    \partial_{i_1}^{k_1} \cdots \partial_{i_r}^{k_r} \left(F_\gamma|_{\Pm_\Sigma}\right)(\Delta) = 0 \text{.}
  \]
  If $r = n = \dim(N_\R)$ and $\rho_{i_1}, \ldots, \rho_{i_n}$ span a cone in $\Sigma$ dual to the vertex $A\in \Delta$, we have 
  \[
    \partial_I \left(F_\gamma|_{\Pm_\Sigma}\right)(\Delta) = \frac{(n+i)!}{i!}f_\gamma(A)\text{.}
  \]
\end{lemma}
\begin{proof}
  To keep notation simple assume $i_j = j$ for $j = 1, \ldots, r$.
  
  For a monomial $\partial_1^{k_1} \cdots\partial_r^{k_r}$, let $\partial_I = \partial_1 \cdots \partial_r$ be the corresponding square free monomial.
  It is enough to show that $\partial_I F_\gamma(\Delta)  = 0$ in order  to prove that $\partial_1^{k_1} \cdots \partial_r^{k_r} F_\gamma(\Delta)  = 0$.

  We compute the expansion of the polynomial $F_\gamma$ at $\Delta$.
  This amounts to expressing $F_\gamma(\Delta+ \sum_{i=1}^s \lambda_i h_i)$ in terms of the monomials $\lambda_1^{\alpha_1} \cdots \lambda_s^{\alpha_s}$ for non-negative integers $\alpha_1, \ldots, \alpha_s$.
  A straightforward computation yields:
  \begin{align*}
    F_\gamma\left( \Delta+\sum_{i=1}^s\lambda_ih_i \right) &= \left( \rho(\Delta) + \sum_{i=1}^s \lambda_i [D_i] \right)^{n+i} \cdot p^*(\gamma)\\
    &= \sum_{\alpha_0 + \alpha_1 + \ldots + \alpha_s = n+i} \binom{n+i}{\alpha_0, \alpha_1, \ldots, \alpha_s} \cdot \rho(\Delta)^{\alpha_0} \cdot [D_1]^{\alpha_1} \cdots [D_s]^{\alpha_s} \cdot p^*(\gamma) \cdot \lambda_1^{\alpha_i} \cdots \lambda_s^{\alpha_s}\\
    &= \frac{(n+i)!}{(n+i-r)!} \cdot \rho(\Delta)^{n+i-r} \cdot [D_1] \cdots [D_r] \cdot p^*(\gamma) \cdot \lambda_1 \cdots \lambda_r + \text{(other terms)}\\
  \end{align*}
  where $\binom{n+i}{\alpha_0,\alpha_1,\ldots,\alpha_s}=\frac{(n+i)!}{\alpha_0! \cdot \alpha_1! \cdots \alpha_s!}$ denotes the usual multinomial coefficient.
  The derivative $\partial_I F_\gamma(\Delta)$ is equal to the coefficient in front of the monomial $\lambda_1 \cdots \lambda_r$ in the expression $F_\gamma(\Delta + \sum_{i=1}^s \lambda_i h_i)$:
  \[
    \partial_I F_\gamma(\Delta)= \frac{(n+i)!}{(n+i-r)!} \cdot \rho(\Delta)^{n+i-r} \cdot [D_1] \cdots [D_r] \cdot p^*(\gamma) \text{.}
  \]
  Since $\Sigma$ is a smooth fan, the divisors $D_1,\ldots, D_r$ intersect transversely in $E_X$.
  So the product $[D_1] \cdots [D_r]$ is the class Poincar\'{e} dual to the set theoretic intersection of these divisors.
  In the case that $e_1, \ldots, e_r$ do not generate a cone in $\Sigma$ the set theoretic intersection of $D_1,\ldots,D_r$ is empty, and so $\partial_I F_\gamma(\Delta)=0$. 

  For the proof of the second part, let $r = n = \dim(N_\R)$.
  We have
   \[
    \partial_I F_\gamma(\Delta)= \frac{(n+i)!}{i!} \cdot \rho(\Delta)^{i} \cdot [D_1] \cdots [D_n] \cdot p^*(\gamma) \text{.}
  \]
 
 If $e_1, \ldots, e_n$ generate a cone in $\Sigma$ dual to the vertex $A$ of $\Delta$, then $[D_1] \cdots [D_n] = E_A$, where $E_A = E \times_T A$ is the torus invariant submanifold in $E_X$ corresponding to $A$.
  In particular, the restriction of  the projection map $p\colon E_A \to B$ is a diffeomorphism.

  Let $h_{\widetilde \Delta}$ be the support function of the polytope $\widetilde \Delta =\Delta - A$ which is the translation of the polytope $\Delta$ for which the vertex $A$ is at the origin.
  Since the vertex of $\widetilde \Delta$ corresponding to $A$ is at the origin, we get
  \[
    h_{\widetilde \Delta}(e_{1}) = \ldots = h_{\widetilde \Delta}(e_{n}) = 0, \qquad \text{and so} \qquad \rho(\widetilde \Delta) = \sum_{j>n}h_{\widetilde \Delta}(e_j) \cdot [D_j].
  \]
  Hence, by the first part, we get $\rho(\widetilde \Delta) \cdot [D_1] \cdots [D_n] = 0$ as there is no cone in $\Sigma$ with more than $n$-rays.
  Therefore:
  \[
    \rho(\Delta)^i \cdot [D_1] \cdots [D_n]  \cdot p^*(\gamma)= \rho(\widetilde \Delta + A)^i \cdot [D_1] \cdots [D_n] \cdot p^*(\gamma)= \rho(A)^i  \cdot E_A \cdot p^*(\gamma) \text{.}
  \]
  By Proposition~\ref{cherneq}, $\rho(A)=p^* c(A)$.
  Since $p\colon E_A\to B$ is a diffeomorphism, we get:
  \[
    \rho(A)^i  \cdot p^*(\gamma)  \cdot E_A =(p^* c(A))^i\cdot p^*(\gamma)  \cdot E_A= c(A)^i\cdot \gamma= f_\gamma(A), 
  \]
  and therefore $\partial_I F_\gamma(\Delta) = \frac{(n+i)!}{i!}f_\gamma(A)$.
\end{proof}

The base case of the inner induction is an immediate corollary of Lemmas~\ref{Ider} and~\ref{Fder}:
\begin{corollary}[Base case of the inner induction]
  \label{sqfree}
  For any $i\le \frac{k}{2}$, any $\gamma \in H^{k-2i}(B,\R)$ and any square free differential monomial $\partial_I=\partial_{i_1}\dots\partial_{i_n}$ of order $n$ (where $I = \{ i_1, \ldots, i_n\} \subseteq \{1, \ldots, s\}$), we have:
  \[
    \partial_I \left( (n+i)!\cdot I_\gamma(\Delta)\right)= \partial_I \left( i!\cdot F_\gamma(\Delta)\right) \text{.}
  \]
\end{corollary}

%%%%%%%%%%%%%%%%%%%%%%%%%%%%%%%%%%%%%%%%%%%%%%%%%%%%%%%%%%%%%%%%%

\subsection{ The inner induction step}
In this subsection, the index set $I \subseteq \{1, \ldots, s \}$ is considered a multiset.
As induction hypothesis, suppose that $\partial_I \left( (n+i)!\cdot I_\gamma(\Delta)\right) = \partial_I \left( i!\cdot F_\gamma(\Delta) \right)$ for any differential monomial $\partial_I$ of multiplicity $m-1\ge0$.
It remains to show that the equality is true for differential monomials of multiplicity $m$.
As before, to keep notation simple assume $i_j=j$ for $j =1, \ldots, r$, so that $\partial_I=\partial_{1}^{k_1}\ldots\partial_{r}^{k_r}$ for positive integers $k_1, \ldots, k_r$.
By reordering the coordinates of $\Pm_\Sigma$, we may assume $k_1>1$. 
By Lemmas~\ref{Ider} and~\ref{Fder}, it suffices to consider the case where the vectors $e_1, \ldots, e_r$ form a cone in $\Sigma$ (as otherwise $\partial_I I_\gamma(\Delta) = 0 = \partial_I F_\gamma(\Delta)$).

The plan is to express $\partial_1$ in terms of a Lie derivative $L_v$ (for some $v \in M_\R$) and other partial derivatives.
Then the (inner) induction step will follow by an explicitly computation of $L_vI_\gamma(\Delta)$ and $L_vF_\gamma(\Delta)$.

As $e_1, \ldots, e_r$ form a cone in the smooth fan $\Sigma$, they can be completed to a basis of $N_\R$.
We take $v$ to be the first vector of the dual basis of $M_\R$.
Then $\langle v, e_1 \rangle = 1$, and  $\langle v, e_j \rangle = 0$ for $j = 2, \ldots, r$.
Since $v\in M_\R$ (considered as an element of $\Pm_\Sigma$) is given by $v = \sum_{i=1}^s \langle v, e_i \rangle h_i$, we obtain that $L_v = \sum_{i=1}^s \langle v,e_i \rangle \partial_i$, and thus $\partial_1 = L_v - \sum_{j>r}\langle v, e_j \rangle \partial_j$.
We get:
\begin{align*}
  \partial_I = \partial_1^{k_1} \ldots \partial_r^{k_r} &= \left( L_v - \sum_{j>r}\langle v, e_j \rangle \partial_j \right) \partial_1^{k_1-1} \cdot \partial_2^{k_2} \cdots \partial_r^{k_r} \\
  &= L_v \cdot \partial_1^{k_1-1} \cdot \partial_2^{k_2} \cdots \partial_r^{k_r} - \sum_{j>r}\langle v, e_j \rangle \cdot \partial_1^{k_1-1} \cdot \partial_2^{k_2} \cdots \partial_r^{k_r} \cdot \partial_j \text{.}
\end{align*}
Since $k_1>1$ and $j>r$, each monomial in the sum $\sum_{j>r}\langle v, e_j \rangle \partial_1^{k_1-1} \cdot \partial_2^{k_2} \cdots \partial_r^{k_r} \cdot \partial_j$ has multiplicity $m-1$, so by the induction hypothesis (of the inner induction), we get
\[
  (n+i)! \cdot \left( \sum_{j>r} \langle v, e_j \rangle \partial_1^{k_1-1} \cdot \partial_2^{k_2} \cdots \partial_r^{k_r} \cdot \partial_j \right) \cdot I_\gamma(\Delta) = i! \cdot \left( \sum_{j>r} \langle v, e_j \rangle \partial_1^{k_1-1} \cdot \partial_2^{k_2} \cdots \partial_r^{k_r} \cdot \partial_j \right) \cdot F_\gamma(\Delta) \text{.}
\]
It remains to consider the first summand:
\begin{lemma}
  \label{Lvind}
  In the situation above, we have $(n+i)! \cdot L_v I_\gamma(\Delta) = i! \cdot L_v F_\gamma(\Delta)$.
\end{lemma}
\begin{proof}
  A direct calculation shows:
  \begin{align*}
    L_v  I_\gamma(\Delta) &= \partial_t \Bigg|_{t=0} \left(  \int_{\Delta+tv} c(x)^{i} \cdot \gamma \diff \mu \right) =  \partial_t\Bigg|_{t=0} \left(  \int_{\Delta} c(x+tv)^{i} \cdot \gamma \diff \mu \right) =\\
    &= \partial_t \Bigg|_{t=0}\left(  \int_{\Delta} \left(\sum_{a=0}^i \binom{i}{a} t^a\cdot c(v)^{a}c(x)^{i-a} \right) \cdot \gamma \diff \mu \right) =  \int_{\Delta}  i\cdot   c(v)c(x)^{i-1}  \cdot \gamma \diff \mu\\
    &= i\cdot\int_{\Delta} c(x)^{i-1}  \cdot  (c(v) \cdot \gamma) \diff \mu = i\cdot I_{c(v)\cdot \gamma}(\Delta) \text{.}
  \end{align*}
  Similarly, by a direct calculation and Proposition~\ref{cherneq}, we get
  \begin{align*}
    L_v F_\gamma(\Delta) &= \partial_t \Big|_{t=0} \Big(\rho(\Delta+tv)^{n+i}\cdot p^*(\gamma)  \Big)=
\partial_t \Big|_{t=0} \Big((\rho(\Delta)+t\rho(v))^{n+i}\cdot p^*(\gamma)  \Big)=\\
    &= \partial_t \Bigg|_{t=0} \left( \left( \sum_{a=0}^{n+i} \binom{n+i}{a} t^a \cdot \rho(v)^a \cdot \rho(\Delta)^{n+i-a}  \right)\cdot p^*(\gamma) \right) = (n+i)\cdot \rho(v) \cdot \rho(\Delta)^{n+i-1} \cdot p^*(\gamma)\\
    &= (n+i)\cdot  \rho(\Delta)^{n+i-1} \cdot p^*(c(v) \cdot \gamma) = (n+i)\cdot  F_{c(v)\gamma}(\Delta).
  \end{align*}
  Since $c(v) \cdot \gamma \in H^{k-2(i-1)}(B,\R)$, by the induction hypothesis of the outer induction, we have
  \[
    (n+i-1)! \cdot I_{c(v) \cdot \gamma}(\Delta) = (i-1)!\cdot F_{c(v) \cdot \gamma}(\Delta).
  \]
  Therefore,
  \[
    (n+i)! \cdot L_v I_\gamma(\Delta) = (n+i)! \cdot  i \cdot I_{c(v) \cdot \gamma}(\Delta) = i! \cdot (n+i) \cdot F_{c(v) \cdot \gamma}(\Delta) =  i! \cdot  L_v F_\gamma(\Delta) \text{.} \qedhere
  \]
\end{proof}

%%%%%%%%%%%%%%%%%%%%%%%%%%%%%%%%%%%%%%%%%%%%%%%%%%%%%%%%%%%%%%%%%

\section{\texorpdfstring{Graded-commutative $n$-self-dual  algebras}{Graded-commutative n-self-dual  algebras}}
\label{commalg}

%%%%%%%%%%%%%%%%%%%%%%%%%%%%%%%%%%%%%%%%%%%%%%%%%%%%%%%%%%%%%%%% 

\subsection {Introduction}
\label{sec:commalg-intro}

Suppose $A$ is a cohomology ring of a compact oriented manifold which can be represented as factor-algebra of some other algebra $B$.
In this section we discuss the question what extra information is needed to determine $A$ out of $B$?
For the reader's convenience, we recall several crucial notions.

\begin{definition} 
  A graded $\K$-algebra $A = A^0 \oplus A^1 \oplus \cdots \oplus A^k \oplus \cdots$ over a field $\K$ is called \emph{graded-commutative} if for any homogeneous elements $x$, $y$, the following relation hold:
  \[
    xy = (-1)^{\deg(x)+\deg(y)} yx
  \]
  where $\deg(x)$ (resp.~$\deg(y)$) denotes the degree of $x$ (resp.~of~$y$).
\end{definition}

Such algebras are fundamental in geometry.
For example, the cohomology ring $A = H^*(X, \K)$ of a CW complex $X$ is known to be a graded-commutative $\K$-algebra.
However, the cohomology ring of compact oriented manifolds satisfies further important properties:

\begin{definition}  
  Let  $n$ be a natural number.
  A graded-commutative $\K$-algebra is \emph{Poincar\'{e} $n$-self-dual} (or just, is \emph{$n$-self-dual}) if the following conditions hold:
  \begin{enumerate}
  \item The algebra $A$ has a multiplicative unit element $e$ which is homogeneous of degree zero, i.e. $e \in A^0 \subseteq A$.
  \item The homogeneous components $A^k$ for $k>n$ vanish, i.e. $A^k=0$ for $k>n$, and $\dim_{\K}(A^n) = 1$.
  \item The pairing $A^k \times A^{n-k} \rightarrow A^n$ induced by multiplication in the algebra $A$ is non-degenerate for $0 \le k \le n$.
  \end{enumerate}
\end{definition}

\begin{remark}
  By the non-degeneracy of the pairing $A^k \times A^{n-k} \to A^n$, we have $\dim_\K(A^k) = \dim_\K(A^{n-k})$ for every $0 \le k \le n$ such that the component $A^k$ is a finite dimensional space.
  In particular, $\dim_\K(A^0) = 1$.
\end{remark}

The following example plays a key role in our discussion.

\begin{example}
  \label{ex2}
  Let $M$ be a connected compact oriented $n$-dimensional manifold.
  By Poincar\'{e} duality, the cohomology ring $A = H^*(M, \K)$ is an $n$-self-dual graded-commutative $\K$-algebra.
  Moreover, $A$ is equipped with two extra structures:
  \begin{enumerate}
  \item A $\K$-linear function $\ell_* \colon A \to \K$  given as follows.
    By linearity, it suffices to define $\ell_*$ on homogeneous elements.
    For homogeneous $\alpha$ of degree $n$, we let $\ell_*(\alpha)$ be equal to the value of the cohomology class $\alpha$ on the fundamental class of the oriented manifold $M$.
    If the degree of $\alpha$ is not $n$, we set $\ell_*(\alpha)=0$.
  \item A $\K$-bilinear intersection form $F_{\ell_*}$ on $A$ defined by the identity $F_{\ell_*}(\alpha, \beta) = \ell_*(\alpha \cdot \beta)$ (recall that we denote the cup product of the cohomology ring $A$ by ``$\cdot$'').
    Note that $F_{\ell_*}$ is non-degenerate (the non-degeneracy of $F_{\ell_*}$ is equivalent to the Poincar\'{e} duality on $A$).
  \end{enumerate}
\end{example}

Consider another graded-commutative $\K$-algebra $B = B^0 \oplus B^1 \oplus \cdots \oplus B^k \oplus \cdots$ with multiplicative unit element $e \in B^0$ and $\dim_{\K}(B^0) = 1$.
Suppose that the cohomology ring $A = H^*(M, \K)$ from Example~\ref{ex2} is a factor-algebra of $B$. 

\begin{question}
  \label{quest1}
  We are going to investigate the following questions:
  \begin{enumerate}
  \item What extra information on $B$ is needed to determined the algebra $A$?
  \item What extra information on $B$ is needed to determined  the intersection form $F_{\ell^*}$ on $A$?
  \end{enumerate}
\end{question}

We provide answers to both questions in Subsections~\ref{n-self-dual} and~\ref{frob} respectively.

%%%%%%%%%%%%%%%%%%%%%%%%%%%%%%%%%%%%%%%%%%%%%%%%%%%%%%%%%%%%%%%%%

\subsection{\texorpdfstring{Graded-commutative $n$-self-dual factor algebras}{Graded-commutative n-self-dual factor algebras}}
\label{n-self-dual}

In this section, let $ B = B^0 \oplus B^1 \oplus \cdots \oplus B^k \oplus \cdots$ be a graded-commutative $\K$-algebra with multiplicative unit element $e \in B^0$ and $\dim_{\K}(B^0) = 1$.
We want to describe all $n$-self dual factor-algebras of $B$.
To that extent we introduce the following notion.

\begin{definition} 
  An ideal $I\subseteq B$  is called \emph{$n$-self-dual} (or \emph{$n$-sd ideal} for short) if $I$ is a two-sided homogeneous ideal and the factor-algebra $A = B/I$ is $n$-self-dual.
\end{definition}

Clearly, factor-algebras of $B$ which are $n$-self-dual are in correspondence with $n$-sd ideals.
Therefore part (1) of Question~\ref{quest1} boils down to describing all $n$-sd ideals in $B$.
Our solution consists of the following steps:

\begin{enumerate}
\item If $I$ is an $n$-sd ideal then $I \cap B^n$ is a hyperplane in $B^n$ (see Lemma~\ref{lemma2}).
\item Conversely, for any hyperplane $L \subseteq B^n$ there is a unique $n$-sd ideal $I = I(L)$ such that $I \cap B^n = L$ (see Lemmas~\ref{Lemma6} and~\ref{Lemma7}).
\item We present an explicit construction of the $n$-sd ideal $I(L)$  (see Definition~\ref{def5}).
\end{enumerate}

\begin{lemma}
  \label{lemma2}
 If  $I\subseteq B$ is an $n$-sd ideal then $L=I\cap B^n$ is a hyperplane in $B^n$.
\end{lemma}
\begin{proof}
  Since $B^n / (I \cap B^n) = A^n$ and $\dim(A^n) = 1$, $L$ is a hyperplane in $B^n$.
\end{proof}

As hyperplanes are in correspondence with linear functions, we introduce the following:

\begin{definition}  
  A  linear function $\ell \colon B\to \K$ on a graded-commutative $\K$-algebra $B$ is \emph{$n$-homogeneous} if it is not identically equal to zero on $B$, but for any $m \ne n$ its restriction to the homogeneous component $B^m$ vanishes.
  Denote by $L_\ell\subseteq B^n$  the hyperplane in $B^n$ defined by identity $L_\ell = \{\ell = 0\} \cap B^n$.
\end{definition}

Clearly, any hyperplane $L \subseteq B^n$ is equal to the hyperplane $L_\ell$ for a unique (up to a non-zero scalar multiple) $n$-homogeneous linear function $\ell$.
Next, we explain how to obtain an $n$-self-dual factor-algebra from an $n$-homogeneous linear function $\ell \colon B \to \K$.

\begin{definition}
  \label{def4} 
  For an $n$-homogeneous  linear function $\ell \colon B \to \K$ let $I_1(L_\ell)$ and $I_2(L_\ell)$ be the subsets of $B$ defined by the following conditions:
  \begin{enumerate}
  \item an element $a \in B$ belongs to $I_1(L_\ell)$ if and only if $\ell(ab) = 0$ for all $b \in B$;
  \item an element $b \in B$ belongs to $I_2(L_\ell)$ if and only if $\ell(ab) = 0$ for all $a \in B$.
  \end{enumerate}
\end{definition}

Clearly, the sets $I_1(L_\ell)$ and $I_2(L_\ell)$ depend only on the hyperplane $L = L_\ell \subseteq B^n$.

\begin{lemma}
  \label{Lemma3}
  For an $n$-homogeneous linear function $\ell$, $I_1(L_\ell) = I_2(L_\ell)$ is a two-sided homogeneous ideal in $B$.
\end{lemma}
\begin{proof} 
  We first prove that the set $I_1(L_\ell)$ is a \emph{right} ideal in $B$.
  If $a \in I_1(L_\ell)$ and $b \in B$, then for any $c \in B$ the identity $\ell( (ab)c ) = \ell ( a(b c) ) = 0$ holds, and thus $ab \in I_1(L_\ell)$.
  If $a, b\in I_1(L_\ell)$ then $\ell(ac) = \ell(bc) = 0$ for any $c\in B$.
  Hence, $\ell( (a+b)c ) =  0$ which implies $a+b \in I_1(L_\ell)$.

  Next we prove that $I_1(L_\ell)$ is homogeneous, i.e. if $a = \sum_{i=0}^d a_i$ for $a_i \in B^i$ is in $I_1(L_\ell)$, then $a_i \in I_1(L_\ell)$ for every $i$.
  For $c = \sum_{m=0}^{d'} c _m$ with $c_m \in B^m$, we have $\ell(a_i c)= \sum_{m=0}^{d'} \ell(a_i c_m)$, and thus it suffices to show $\ell(a_i c_m) = 0$ for any homogeneous $c_m \in B_m$.
  Since $\ell$ is $n$-homogeneous, $\ell(a_i c_m) = 0$ for every $i$ with $i+m\neq n$.
  However, from $0 = \ell(ac_m) = \sum_{i=0}^d \ell(a_ic_m)$, it follows $\ell(a_{n-m} c_m) = 0$.
  Thus $a_i \in I_1(L_\ell)$.
  
  Analogously, it follows that $I_2(L_\ell)$ is a homogeneous \emph{left} ideal in $B$.

  Let us show that $I_1(L_\ell) \subseteq I_2(L_\ell)$.
  Let  $a = \sum_{j=0}^d a_j$ be in $I_1(L_\ell)$ and  $b = \sum_{i=0}^{d'} b_i$ be in $B$ where $a_i, b_i \in B^i$.
  Then $\ell(ba) = \sum_{i=0}^{d'} \sum_{j=0}^d \ell(b_i a_j) = \sum_{i=0}^{d'}\sum_{j=0}^d (-1)^{i+j} \ell( a_jb_i)$.
  Since $I_1(L_\ell)$ is homogeneous, we have $a_j \in I_1(L_\ell)$, and thus $\ell( a_jb_i)=0$ for all couples $i, j$.
  It follows $\ell(ba) = 0$ which means that $I_1(L_\ell) \subseteq I_2(L_\ell)$.
  
  Similarly, we can prove that $I_2(L_\ell)\subseteq I_1(L_\ell)$ and the statement is proven.
\end{proof}

\begin{definition}
  \label{def5}
  For an $n$-homogeneous linear function $\ell$, let $I(L_\ell)$ be the two-sided ideal $I_1(L_\ell) = I_2(L_\ell)$.
  The ideal $I(L_\ell)$ depends only on the hyperplane $L = L_\ell$, and so the notation $I(L)=I(L_\ell)$ makes sense.
\end{definition}

\begin{lemma}
  \label{Lemma6}
  For any hyperplane $L \subseteq B^n$ the set $I(L) \subseteq B$ is an $n$-sd ideal.
\end{lemma}
\begin{proof} 
  By Lemma~\ref{Lemma3}, $I(L)$ is a two-sided homogeneous ideal.
  Set $A=B/I(L)$.
  
  If $k > n$, then for any $b_k \in B^k$ and any $c = \sum_{i=0}^d c_i$ in $B$ with $c_i \in B^i$, we have $\ell(b_k c) = \sum_{i=0}^d \ell(b_k c_i) = 0$ (since $\ell$ is $n$-homogeneous).
  Hence, $B^k = I(L) \cap B^k$ for $k>n$, or equivalently $A^k$ vanishes for $k>n$.
  
  It remains to show that the pairing $A^k \times A^{n-k} \to A^n$ induced by multiplication is non-degenerate.
  Suppose $a \in A^k$ is in the kernel of this pairing, i.e. $ab = 0$ for any $b \in A^{n-k}$.
  Let $a^* \in B^k$ be an element whose image in $A$ is equal to $a$.
  Then for any $b^* \in B^{n-k}$, we have $a^*b^* \in I(L)$, and in particular, $\ell(a^*b^*) = 0$.
  Since $\ell$ is $n$-homogeneous, for any $c \in B$, we have $\ell(a^*c) = \ell(a^*c_{n-k})$ where $c_{n-k} \in B^{n-k}$ is the homogeneous component of $c$ of degree $n-k$.
  Thus $a^* \in I_1(L_\ell) = I(L)$, or equivalently $a=0$ in $A$, i.e. the pairing is non-degenerate in the first component.
  A similar argument shows non-degeneracy in the second component.
\end{proof}

\begin{lemma}
  \label{Lemma7}
  For any hyperplane $L\subseteq B^n$ there is a unique $n$-sd ideal $I\subseteq B$ such that $I\cap B^n=L$.
  This ideal $I$ coincides with the ideal $I(L)$.
\end{lemma}
\begin{proof} 
  Let $\ell$ be an $n$-homogeneous linear function such that $L=L_\ell$.
  
  We start by showing that $I(L) \cap B^n = L$.
  The inclusion ``$\subseteq$'' is straightforward.
  For the reverse inclusion suppose $a \in L$.
  Then for any $b = \sum_{i=0}^d b_i$ in $B$ with $b_i \in B^i$, we have $\ell(ab) = \sum_{i=0}^d \ell(ab_i) = \ell(ab_0) = b_0 \ell(a) = 0$, and thus $a \in I(L) \cap B^n$.
  This shows the existence-part of the statement.
  
  To show uniqueness, let $I \subseteq B$ be an $n$-sd ideal with $I \cap B^n = L$.
  As both $I$ and $I(L)$ are homogeneous, it suffices to show that $I \cap B^k = I(L) \cap B^k$ for every $k$.
  If $a\in I(L) \cap B^k = I_1(L_\ell)\cap B^k$, then its image in the factor-algebra $A=B/I$ belongs to the kernel of the pairing $A^k\times A^{n-k}\to A^n = B/L$.
  Thus $a\in I \cap B^k$ and $I_1(L_\ell) \cap B^k \subseteq I \cap B^k$.
  Conversely, if $a\notin I_1(L_\ell)\cap B^k$, then there is $b \in B^{n-k}$ such that $ab \notin L$.
  Thus the image of the element $ab \in B^n$ is not equal to zero in $A = B/I$ and $a \notin I \cap B^k$.
\end{proof}

%%%%%%%%%%%%%%%%%%%%%%%%%%%%%%%%%%%%%%%%%%%%%%%%%%%%%%%%%%%%%%%%%

\subsection{\texorpdfstring{Frobenius forms and $n$-self-dual factor-algebras}{Frobenius forms and n-self-dual factor-algebras}}
\label{frob}

In this section, we investigate the second part of Question~\ref{quest1}.
For the reader's convenience, we provide a summary at the end of the section.

Consider a bilinear form $F\colon L_1 \times L_2 \to \K$ on a product of $\K$-linear spaces $L_1$ and $L_2$.
Recall the \emph{left radical $R_1\subseteq L_1$} respectively the \emph{right radical $R_2\subseteq L_2$} of $F$:
\begin{enumerate}
    \item $a\in R_1$ if and only if $F(a,b)=0$ for all $b\in L_2$;
    \item $b\in R_2$ if and only if $F(a,b)=0$ for all $a\in L_1$.
\end{enumerate}

\begin{definition}
  Let $\ell \colon B \to \K$ be an $n$-homogeneous linear function on a graded-commutative algebra $B$.
  The \emph{Frobenius bilinear form} $F_\ell \colon B \times B \to \K$ associated with $\ell$ is the form defined by the identity $F_\ell(a,b) = \ell(a \cdot b)$.
\end{definition}

\begin{theorem}
  \label{Th8}
  Let $\ell\colon B \to \K$ be an $n$-homogeneous linear function on a graded-commutative algebra $B$.
  \begin{enumerate}
    \item The $n$-sd ideal $I(L_\ell) = I_1(L_\ell) = I_2(L_\ell)$ coincides with the left and also the right radical of  $F_\ell$.
    \item There is a unique $n$-homogeneous linear function $\ell_*$ on the factor algebra $A=B/I(L)$ such that $\ell=\rho^* \ell_*$ where $\rho \colon B \to A$ is the natural homomorphism.
      Furthermore, for any $a, b\in B$ the relation $F_\ell(a,b) = F_{\ell_*}(\rho(a),\rho(b))$ holds.
    \item The Frobenius form $F_{\ell_*}$ on $A$ is non-degenerate.
  \end{enumerate}
\end{theorem} 
\begin{proof}
  By definition the left radical (resp.~the right radical) of the form $F_\ell$ coincides with $I_1(L_\ell)$ (resp.~with $I_2(L_\ell)$).
  By Lemma~\ref{Lemma3}, $I(L_\ell)=I_2(L_\ell)=I_2(L_\ell)$, and thus the first statement follows (see also Definition~\ref{def5}).
  
  For the second statement, observe that the kernel $L\subseteq B^n$ of the restriction of the function $\ell|_{B^n}$ coincides with the kernel of the surjective map $\rho\colon B^n\to A^n$ (see Lemma~\ref{Lemma7}).
  Thus there is a unique function $\ell_*\colon A^n\to \K$ such that $\ell=\rho^* \ell_*$.
  The relation $F_\ell(a,b)=F_{\ell^*}(\rho(a), \rho(b))$ follows by definition.

  For the last statement, observe that the form $F_{\ell_*}$ is obtained from the form $F_\ell$ by taking the quotient of $B$ by the left and right radical of the form $F_\ell$.
  Thus $F_{\ell_*}$ is a non-degenerate form.
\end{proof}

Let us summarize the results from this section.
In order to reconstruct an $n$-sd factor-algebra $A=B/I$ equipped with an intersection form on $A$ it suffices to fix a Frobenius form $F_\ell$ on $B$ corresponding to an $n$-homogeneous linear function $\ell$ on $B$.
The algebra $A$ is equal to $B/I(L_\ell)$ (note that the algebra $A$ depends only on the hyperplane $L = L_\ell$, i.e.  proportional $n$-homogeneous functions define the same algebra $A$).
The form $F_\ell$ is induced by the unique non-degenerate intersection form $F_{\ell_*}$ on $A$.
The set of $n$-sd factor-algebras $A=B/I$ equipped with non-degenerate intersection forms is in one-to-one correspondence with the set of $n$-homogeneous linear functions on $B$.

\begin{theorem}
  \label{thm-nsdquatient}
  Let $B$ be a graded-commutative algebra over a field $\K$, and let $A$ be an $n$-sd factor-algebra of $B$ with a chosen isomorphism $\phi\colon A^n\to \K$.
  Let $\rho\colon B\to A$ be the natural homomorphism.
  Then
  \[
    A\simeq B/I(L_\ell),
  \]
  where $\ell$ is the $n$-homogeneous linear function defined by
  \[
    \ell\colon B^n\to \K, \quad \ell(b)=\phi(\rho(b)),
  \]
  and extended by 0 to $B^i$ with $i\ne n$.

  Moreover, for any $n$-homogeneous linear function $\ell$, the algebra $A=B/I(L_\ell)$ is $n$-self dual and the corresponding Frobenius form $F_{\ell_*}$ on $A$ is non-degenerate.
\end{theorem}

%%%%%%%%%%%%%%%%%%%%%%%%%%%%%%%%%%%%%%%%%%%%%%%%%%%%%%%%%%%%%%%%%

\section{Applications}
\label{sec:applications}

%%%%%%%%%%%%%%%%%%%%%%%%%%%%%%%%%%%%%%%%%%%%%%%%%%%%%%%%%%%%%%%% 

\subsection{The Cohomology ring of toric bundles}\label{sec:cohtoricbundle}
In this section, we apply the results from Sections~\ref{BKKproof} and~\ref{commalg} to give a description of the cohomology ring of a toric bundle.
Let, as before, $p\colon E\to B$ be a principal torus bundle, $\Sigma$ a smooth projective fan with $\Sigma(1) = \{\rho_1, \ldots, \rho_r\}$, $X = X_\Sigma$ the corresponding toric variety, and $E_X$ the corresponding toric bundle.
By the Leray-Hirsch theorem (see Theorem~\ref{thm:Leray-Hirsch}) the cohomology ring $H^*(E_X,\R)$ is a quotient of the polynomial algebra $R[x_1,\ldots,x_r]$, where $R=H^*(B,\R)$ is the cohomology ring of the base.
\begin{theorem}
  \label{cohbundle}
  In the notation from above the cohomology ring of $E_X$ is given by
  \[
    H^*(E_X,\R) \simeq R[x_1,\ldots,x_r]/I(L_\ell)
  \]
 where $\ell \colon R[x_1,\ldots,x_r]\to \R$ is a $(k+2n)$-homogeneous linear function defined by: 
  \[
    \ell(\gamma \cdot x_{i_1} \cdots x_{i_s}) = I_\gamma(\rho_{i_1},\ldots, \rho_{i_s})
  \]
  for any monomial $\gamma\cdot x_{i_1}\cdots x_{i_s}$ with $\deg(\gamma)+2s=k+2n$.
\end{theorem}
\begin{proof}
  By Theorem~\ref{thm-nsdquatient}, as $H^*(E_X,\R)$ is a graded commutative $(k+2n)$-self dual factor algebra of $R[x_1,\ldots,x_r]$, it is given by $R[x_1,\ldots,x_r]/I(L_\ell)$ for a $(k+2n)$-homogeneous linear function $\ell$ which is obtained by pairing cohomology classes with the fundamental class of $E_X$.
  Hence the statement follows by Theorem~\ref{BKK}.
\end{proof}

We conclude this subsection with a proof of Theorem~\ref{thm:US}, i.e. a proof of the Sankaran-Uma-description of the cohomology ring of toric bundles.
\begin{proof}[Proof of Theorem~\ref{thm:US}]
  As above, let $R$ denote the cohomology ring of the base.
  We want to show that the cohomology ring $H^*(E_X, \R)$ is isomorphic to the quotient $R[x_1, \ldots, x_r]/(I+J)$, where 
  \[
    I = \left\langle x_{j_1} \cdots x_{j_k} \colon \rho_{j_1}, \ldots, \rho_{j_k} \; \text{do not span a cone of} \; \Sigma \right\rangle \text{,} \qquad
    J=\left\langle - c\left( \lambda \right) + \sum_{i=1}^r \langle e_i, \lambda \rangle x_i \colon \lambda \in M \right\rangle  \text{.}
  \]
  By the Leray-Hirsch theorem (Theorem~\ref{thm:Leray-Hirsch}), there is a surjection $\phi\colon R[x_1, \ldots, x_r] \to H^*(E_X, \R)$.
  By  Proposition~\ref{cherneq}, the ideal $J$ is contained in $\ker(\phi)$.
  Furthermore, by Theorem~\ref{cohbundle}, the polarisation formula~\eqref{eq:polarisation_formula}, and the first part of Lemma~\ref{Ider}, we have $I\subseteq \ker(\phi)$.
  Thus, we have a surjection
  \[
    \phi\colon R[x_1, \ldots, x_r]/(I+J) \to H^*(E_X, \R)\text{.}
  \]
  To simplify the notation, we also denote this surjection with $\phi$.
  It remains to show that the surjection $\phi$ is an isomorphism, i.e. $\phi$ is injective.
  For that it suffices to verify that $\dim_\R R[x_1, \ldots, x_r]/(I+J) = \dim_\R H^*(E_X, \R)$.

  Let $SR_\Sigma \coloneqq \R[x_1, \ldots, x_r]/I$ be the Stanley-Reisner algebra associated to $\Sigma$.
  Then
  \[
    A \coloneqq  R[x_1, \ldots, x_r]/(I+J) = (R\otimes SR_\Sigma) / J.
  \]
  We consider the ring  $A_0 \coloneqq R\otimes (SR_\Sigma / J_0)$ where $J_0$ is the ideal given by $J_0 = \big\langle \sum_{i=1}^r \langle e_i, \lambda \rangle x_i \colon \lambda \in M \big\rangle$.
  By the K\"unneth formula, we have that
  \[
    A_0 = R\otimes (SR_\Sigma/J_0) = R\otimes H^*(X, \R) \simeq H^*(B\times X,\R).
  \]
  By the Leray-Hirsch theorem (Theorem~\ref{thm:Leray-Hirsch}), we have $\dim_\R (R \otimes H^*(X,\R)) = \dim_\R H^*(E_X,\R)$, and thus
  \[
    \dim_\R A_0 =\dim_\R H^*(B\times X,\R) = \dim_\R H^*(E_X,\R) \leq \dim_\R A.
  \]
  Hence it suffices to show that $\dim_\R A_0 \geq \dim_\R A$.
  We have $A=(R\otimes SR_\Sigma) / J$ and $A_0=(R\otimes SR_\Sigma) / J_0$, with $J$ and $J_0$ homogeneous ideals.
  So to show that $\dim_\R A_0 \geq \dim_\R A$ it suffices to verify
  \[
    \dim_\R J^{i} \geq  \dim_\R J_0^{i}, 
  \]
  where $J^{i}= (R\otimes SR_\Sigma)^i\cap J$, $J_0^{i}= (R\otimes SR_\Sigma)^i\cap J_0$ are the $i$-th homogeneous components of $J$ and $J_0$ respectively.
  
  Denote by $s_\lambda = - c\left( \lambda \right) + \sum_{i=1}^r \langle e_i, \lambda \rangle x_i$ and $s_\lambda^0 = \sum_{i=1}^r \langle e_i, \lambda \rangle x_i$ the generators of the ideals $J$ and $J_0$ respectively.
  Note $R\otimes SR_\Sigma$ is a bi-graded ring and $s_\lambda = s_\lambda^0 - c(\lambda)$ is the difference of two bi-homogeneous elements with $\deg(c(\lambda))=(2,0), \deg(s_\lambda^0) = (0,2)$.
  Suppose $v_1, \ldots, v_N \in J_0$ is an $\R$-basis of $J_0^i$.
  For all $k=1, \ldots, N$, we may assume $v_k$ is bi-homogeneous of degree $(n_k,i-n_k)$ for some $n_k = 0, 1, \ldots, i$.
  Hence,
  \[
    v_k = \sum_{\lambda \in M} r_{\lambda,k}^{(n_k,i-n_k-2)} s^0_\lambda \qquad\text{for bi-homogeneous $r_{\lambda,k}^{(n_k,i-n_k-2)} \in R\otimes SR_\Sigma$ of degree $(n_k,i-n_k-2)$.}
  \]
  Set $w_k \coloneqq \sum_{\lambda\in M} r_{\lambda,k}^{(n_k,i-n_k-2)}s_\lambda \in J^i$ for $k=1, \ldots, N$ and note $w_k = v_k + w_k'$ for bi-homogeneous $w_k' \in R\otimes SR_\Sigma$ of degree $(n_k+2,i-n_k-2)$.
  We claim $w_1, \ldots, w_N$ are $\R$-linearly independent concluding the proof.
  
  Suppose $(a_1, \ldots, a_N) \neq (0, \ldots, 0) \in \R^N$ such that $\sum_{i=1}^N a_i w_i = 0 \in J^i$.
  Set $n \coloneqq \min\{ n_k \colon k=1, \ldots,N, a_k \neq0 \}$.
  Then the \mbox{$(n,i-n)$-homogeneous component} of $\sum_{i=1}^N a_i w_i$ is given by $\sum_{k \in \mathcal{K}} a_k v_k=0$ for some subset $\mathcal{K} \subseteq \{1, \ldots, N\}$.
  A contradiction (as $v_1, \ldots, v_N$ form a basis of $J_0^i$).
\end{proof}
%%%%%%%%%%%%%%%%%%%%%%%%%%%%%%%%%%%%%%%%%%%%%%%%%%%%%%%%%%%%%%%% 

\subsection{Ring of conditions} 
\label{sec:roch}

The ring of conditions constitutes an intersection ring for horospherical homogeneous spaces.
We refer to \cite{DP85} for further details and references.
The exposition given here follows \cite{roch}.

Let $G$ be a reductive complex algebraic group and $H \subseteq G$ a closed subgroup.
A cycle of codimension $k$ of $G/H$ is a formal linear combination $k_1 X_1 + \ldots + k_r X_r$ of irreducible closed subvarieties $X_i$ of $G/H$ with $k_i \in \Z$.
We write $\Zm^k(G/H)$ for the $\Z$-module of cycles of codimension $k$.
Two cycles $X,Y$ are said to \emph{intersect properly} if either their intersection is empty or the dimension of each irreducible component of $X \cap Y$ equals $\dim(X) + \dim(Y) - \dim(G/H)$.
We need the following result by Kleiman \cite[Corollary 4]{Kl74}:

\begin{theorem}[Kleiman's transversality theorem]
  \label{thm:kleiman}
  For any two irreducible subvarieties $X, Y \subseteq G/H$ there exists a dense open $U \subseteq G$ such that $g\cdot X$ and $Y$ intersect properly for each $g \in U$.
  If $X, Y$ have complementary dimensions, $g\cdot X \cap Y$ is a finite union of points and $|g \cdot X \cap Y|$ is constant for generic $g \in G$.
\end{theorem}

Theorem~\ref{thm:kleiman} gives rise to an intersection pairing between algebraic cycles of complementary codimensions:
\[
  \Zm^k(G/H) \times \Zm^{\dim(G/H) - k}(G/H) \to \Z; \quad(X,Y) \mapsto (X \cdot Y) \coloneqq |g \cdot X \cap Y| \quad (\text{for generic} \; g \in G) \text{,}
\]
where $X,Y \subseteq G/H$ are assumed to be irreducible subvarieties.

Two algebraic cycles $X,Y \in \Zm^k(G/H)$ are said to be equivalent if for any algebraic cycle of complementary codimension $Z \in \Zm^{\dim(G/H)-k}(G/H)$ the intersection products coincide, i.e. $(X \cdot Z) = (Y \cdot Z)$.
We denote the group of equivalence classes by $C^k(G/H)$ (``the group of conditions of codimension $k$'').
The intersection pairing factors through the equivalence relation, so that we obtain a pairing: $C^k(G/H) \times C^{\dim(G/H)-k}(G/H) \to \Z$.

De Concini and Procesi showed that $C^*(G/H) \coloneqq \bigoplus_{k=0}^{\dim(G/H)} C^k(G/H)$ can be equipped with a ring structure:
\begin{theorem}[{\cite[Section 6.3]{DP85}}]
  \label{thm:DP}
  Suppose $G/H$ is a horospherical homogeneous space.
  Define an intersection product on $C^*(G/H)$: $[X] \cdot [Y] \coloneqq [g X \cap Y]$ for generic $g \in G$ where $X,Y \subseteq G/H$ are irreducible subvarieties.
  This intersection product is well-defined and there is a canonical isomorphism of graded rings
  \[
    C^*(G/H) = \varinjlim_E H^*(E, \Z)
  \]
  where the limit is taken over complete (or equivalently smooth) horospherical embeddings $G/H \hookrightarrow E$.
\end{theorem}
Theorem~\ref{thm:DP} is a special case of \cite[Section 6.3]{DP85} which suffices for our purposes.
Furthermore, the intersection product in Theorem~\ref{thm:DP} might not be well-defined for arbitrary homogeneous spaces $G/H$.
We refer to \cite[Exercise 4.1]{roch} for a counterexample.
In this paper we will exclusively work with the ring of conditions with real coefficients.
Therefore, we introduce the notation $C^*_\R(G/H)$ for the tensor product $C^*(G/H) \otimes_\Z \R$.

For the rest of this section let $H \subseteq G$ be a horospherical subgroup.
In Section~\ref{sec:horospherical}, we saw that $P = N_G(H)$ is a parabolic subgroup and that $E = G/H$ is a principal torus bundle over $G/P$ with respect to the torus $T = P/H$.
For a $T$-toric variety $X$, we saw that the associated toric bundle $E_X$ is a horospherical variety.
Since any toric variety is dominated by a smooth projective one, the same is true for these horospherical varieties $E_X$.
Furthermore, using the Luna-Vust theory of spherical embeddings (see, for instance, \cite[Theorem~3.3]{Knop:SphEmbd}), it follows that an (arbitrary) horospherical variety $G/H \hookrightarrow Y$ is dominated by some horospherical variety $E_X$ as described above.
Thus, the smooth projective $E_X$ form a cofinal set for the direct limit in Theorem~\ref{thm:DP}.

We now give a description of the ring of conditions $C^*(G/H)$ using Theorem~\ref{cohbundle}.
Let $R=H^*(G/P,\Z)$ be the cohomology ring of the flag variety $G/P$ and $\Sym^*(\Pm)$ be the symmetric algebra of integral virtual polytopes.
Furthermore set $k=\dim_\C(G/P)$ and $n=\dim_\C(P/H)$.

\begin{corollary}
  \label{cor:roch}
  With the notation above, $C^*(G/H)$ is an $2(n+k)$-self dual factor algebra of $R\otimes \Sym(\Pm)$:
  \[
    C^*(G/H)  =R\otimes \Sym(\Pm)/I_\ell, \; \text{with} \; \ell(\gamma\otimes \Delta_{i_1}\cdots\Delta_{i_s}) = I_\gamma(\Delta_{i_1},\ldots,\Delta_{i_s}) \text{.}
  \]
  In particular, $\ell(\gamma\otimes \Delta_{i_1}\cdots\Delta_{i_s}) =0$ unless $\deg(\gamma)+2s=2(n+k)$.
\end{corollary}

\begin{remark}
  Note that in Theorem~\ref{cohbundle} we compute the cohomology ring of toric bundles over $\R$.
  Since spherical varieties come equipped with a $\C^*$-action that has only finitely many fixed points, by the Bia{\l}ynicki-Birula decomposition \cite{BB}, we have a paving by affine spaces.
  Thus $H^*(E_X,\Z)$ is torsion free in this case, and hence our calculations also work over $\Z$. 
\end{remark}

We conclude this section with an alternative description of $C^*_\R(G/H)$ using the description of the cohomology ring from Section~\ref{sec:horospherical}.
The challenge is to find a good description of the ring in Theorem~\ref{thm:US} as we want to take the direct limit over all smooth projective $E_X$.
The following approach is inspired by \cite{Br96} (see also \cite{roch}).

As in Section~\ref{sec:horospherical}, let $M$ be the character lattice of the torus $T=P/H$ and set $M_\R = M \otimes \R$ and $N_\R = \Hom_\Z(M, \R)$.
Let $\Sigma$ be a smooth projective fan in $N_\R$.
A map $f \colon N_\R \to \R$ is \emph{piecewise polynomial} if for any $\sigma \in \Sigma$, the map $f|_\sigma \colon \sigma \to \R$ extends to a polynomial function on the linear space $\lspan_\R\{ \sigma\}$, i.e. a piecewise polynomial function $f$ on $\Sigma$ is a collection of compatible polynomial functions $f_\sigma \colon \sigma \to \R$.
In particular, such a function is continuous.
We denote by $\PP_\Sigma$ the set of all piecewise polynomial functions on $\Sigma$ which is a ring under pointwise addition and multiplication.
Let $\Sym(M_\R)$ be the symmetric algebra of polynomial functions on $N_\R$.
Note that $\PP_\Sigma$ is a positively graded $\R$-algebra with graded subalgebra $\Sym(M_\R)$.
Indeed, any piecewise polynomial function uniquely decomposes into a sum of homogeneous piecewise polynomial functions.

We can now reformulate Theorem~\ref{thm:US}:
\begin{proposition}[{\cite[Proposition 4.10]{roch}}]
  If $E_X$ is a horospherical variety obtained as associated toric bundle of a smooth projective toric variety $X=X_\Sigma$, the cohomology ring $H^*(E_X,\R)$ is isomorphic as an $H^*(G/P, \R)$-algebra to the quotient of $H^*(G/P,\R) \otimes \PP_\Sigma$ by the ideal $\left\langle c(\lambda) \otimes 1 - 1 \otimes \langle \cdot, \lambda \rangle \colon \lambda \in M \right\rangle$.
\end{proposition}

Let $\PP$ be the set of all piecewise polynomial functions on smooth projective fans in $N_\R$, i.e., $\PP = \bigcup_\Sigma \PP_\Sigma$ where the union is taken over all smooth projective fans $\Sigma \subseteq N_\R$.

\begin{theorem}[{\cite[Theorem 4.11]{roch}}]\label{thm:condjoh}
  We have that
  \[
    C^*_\R(G/H)  \simeq \left( H^*(G/P,\R) \otimes \PP \right) / \left\langle c(\lambda)\otimes 1 - 1\otimes \langle \cdot, \lambda \rangle \colon \lambda \in M \right\rangle \text{,}
  \]
  where $\langle \cdot, m \rangle \in \Sym(M_\R)$ is a globally linear function on $N_\R$.
\end{theorem}

\begin{remark}
  Such a description of the ring of conditions as in Theorem~\ref{thm:condjoh} was certainly known to De Concini and Procesi although we couldn't find any references.
  Their approach relied on a computation of equivariant cohomology of toroidal compactifications of horospherical homogeneous spaces.
\end{remark}

%%%%%%%%%%%%%%%%%%%%%%%%%%%%%%%%%%%%%%%%%%%%%%%%%%%%%%%%%%%%%%%%%

\section{\texorpdfstring{Cohomology rings of toric bundles with $H^*(B,\R)$ generated in degree 2}{Cohomology rings of toric bundles with H*(B,R) generated in degree 2}} \label{sec:base-var}
The classical \BKK theorem is concerned with intersections of divisors in toric varieties: it computes the self-intersection polynomial on $H^2(X_\Sigma)$, i.e. $f([D]) = [D]^{\dim{X_\Sigma}}$ for divisors $D \subseteq X_\Sigma$.
Here we use our generalisation of \BKK to compute the self-intersection polynomial on the second cohomology group of a toric bundle.
We then use this to compute the cohomology ring of toric bundles for which $H^*(B,\R)$ is generated in degree 2. 

\subsection{\texorpdfstring{Self-intersection polynomial on $H^2(E_X)$}{Self-intersection polynomial on H2(EX)}}
The Leray-Hirsch theorem (Theorem~\ref{thm:Leray-Hirsch}) yields a surjection
\[
    \bar \rho\colon H^2(B,\R)\oplus \Pm_\Sigma \to H^2(E_X,\R),\qquad (\gamma, \Delta) \mapsto p^*(\gamma) + \rho(\Delta).
\] 
Therefore, it suffices to compute the self-intersection polynomial $f\colon H^2(B,\R)\oplus \Pm_\Sigma \to \R$ given by $f(x)=0$ if $B$ has odd (real) dimension and $f(x) = \bar \rho(x)^{n+s}$ if the (real) dimension of $B$ is even, say $k=2s$.
Clearly, the case of even dimensions $\dim_\R B = k = 2s$ needs further consideration.

Let $W\subseteq V$ be two $\R$-vector spaces, and $\Sigma\subseteq W^*$ be a fan.
Let us denote by $\Am_{W,\Sigma}^+$ the \emph{cone of $W$-affine polytopes}, i.e. the cone of polytopes $\Delta\subseteq V$ whose affine span is parallel to $W$ such that the normal fan of a translation of $\Delta$ into $W$ is a coarsening of $\Sigma$.
Let us denote by $\Am_{W,\Sigma}$ the linear subspace of the vector space of virtual polytopes in $V$ generated by $\Am_{W,\Sigma}^+$.
Similarly $\Am_W^+$ (resp.~$\Am_W$) denotes the cone of all polytopes (resp.~the subspace of all virtual polytopes) in $V$ whose affine span is parallel to $W$.
Every virtual polytope $\Delta\in \Am_W$ admits a (non-unique) representation $\Delta =\Delta_0 + v$ for some virtual polytope $\Delta_0$ in $W$ and $v\in V$. 
\begin{remark}\label{rem:vwell}
    Although the above representation of $\Delta$ as sum $\Delta_0 + v$ is not unique, the image of $v$ under the natural projection $\pi \colon V \to V/W$ is well-defined.
    In particular, a choice of a complementary subspace $W^\perp\subseteq V$ to $W$ induces a decomposition of vector spaces:
    \[
        \Am_{W,\Sigma} = \Pm_{W,\Sigma} \oplus W^\perp \quad \text{and} \quad \quad \Am_W = \Pm_W \oplus W^\perp.
    \]
\end{remark}

Let us fix a Lebesgue measure $\mu$ on $W$.
We can regard $\mu$ as a translation invariant Lebesgue measure on any affine subspace in $V$ parallel to $W$.
Thus any smooth function $g\colon V \to \R$ gives rise to an integral functional:
\[
    I_g \colon \Am_W \to \R, \quad \Delta \mapsto \int_\Delta g \diff \mu.
\]
In particular, if $g\colon V\to \R$ is a homogeneous polynomial of degree $d$, then, by Theorem~\ref{ultimatepoly} and Remark~\ref{rem:vwell}, $I_g$ is a homogeneous polynomial on $\Am_{W}$ of degree $d+\dim(W)$. 

Let $p \colon E\to B$ be a torus bundle and let $\Sigma$ be a smooth projective fan with associated toric  variety $X$ and corresponding toric bundle $E_X$.
Consider a pair of vector spaces:
\[
    W=M_\R\subseteq M_\R\oplus H^2(B,\R) = V.
\]
As before, we have a linear map $ \bar\rho \colon \Am_{M_\R,\Sigma} \to H^2(E_X,\R)$ given by
\[
    \bar\rho(\Delta) = \rho(\Delta_0) + p^*(\gamma),
\]
where $\Delta=\Delta_0+\gamma$ for a virtual polytope $\Delta_0$ in $M_\R$ and $\gamma \in  H^2(B,\R)$.
Note, by Remark~\ref{rem:vwell}, the representation of $\Delta$ as sum $\Delta_0+\gamma$ is unique and hence the map $\bar\rho$ is well defined.
Furthermore, we have a map $ \bar c\colon V\to H^2(B,\R)$:
\[
    \bar c(\lambda, \gamma) = c(\lambda) +\gamma.
\]

\begin{theorem}\label{thm:BKKdeg2}
    In the notation as above we have
    \[
        s!\cdot\bar\rho(\Delta)^{n+s} = (n+s)!\cdot\int_\Delta \bar c(x)^s \diff \mu
    \]
    for any $\Delta\in\Am_{M_\R,\Sigma}$.
\end{theorem}

\begin{proof}
    By Remark~\ref{rem:vwell}, we can write $\Delta = \Delta_0 + \gamma$ for unique $\Delta_0\in \Pm_{M_\R,\Sigma}$ and $\gamma\in H^2(B,\R)$.
    Thus we have 
    \[
         \bar\rho(\Delta)^{n+s} = (\rho(\Delta_0) + p^*(\gamma))^{n+s} = \sum_{i=0}^{s} \binom{n+s}{i} \underbrace{p^*(\gamma)^i \cdot \rho(\Delta_0)^{n+s-i}}_{F_{\gamma^i}(\Delta_0)} = \frac{(n+s)!}{s!}\cdot \sum_{i=0}^s \binom{s}{i} I_{\gamma^i}(\Delta_0),
    \]
    where in the last equality follows by Theorem~\ref{BKK}.
    On the other hand we have
    \[
        \int_\Delta \bar c(x)^s \diff \mu = \int_{\Delta_0+\gamma} \bar c(x)^s \diff \mu = \int_{\Delta_0} \bar c(x+\gamma)^s \diff \mu = \int_{\Delta_0} (c(x)+\gamma)^s \diff \mu = \sum_{i=0}^s \binom{s}{i} I_{\gamma^i}(\Delta_0).
    \]
    Therefore, we obtain
    \[
        s! \cdot \bar \rho(\Delta)^{n+s} = (n+s)! \cdot \sum_{i=0}^s \binom{s}{i} I_{\gamma^i}(\Delta_0) = (n+s)!\cdot \int_\Delta \bar c(x)^s \diff \mu. \qedhere
    \]
\end{proof}

\subsection{\texorpdfstring{Cohomology rings of toric bundles with $H^*(B)$ generated in degree 2}{Cohomology rings of toric bundles with H*(B) generated in degree 2}}
Let us discuss a description of the cohomology ring of toric bundles when $H^*(B)$ is generated in degree $2$.
Then all odd degrees of $H^*(B, \R)$ vanish, and thus the (real) dimension of $B$ is even, say $k = 2s$.
Furthermore, the cohomology ring $H^*(E_X, \R)$ of the toric bundle $E_X$ associated to a toric variety $X = X_\Sigma$ for a smooth projective fan $\Sigma$ is also generated in degree $2$ (use, e.g., the Leray-Hirsch Theorem, see Theorem~\ref{thm:Leray-Hirsch}).
If the complex dimension of the torus is $n$, then the real dimension of $E_X$ is even, namely $2(n+s)$.

Since all odd degrees of $H^*(E_X, \R)$ vanish, we consider the grading $H^{*/2}(E_X, \R)$, i.e., we consider degrees divided by $2$.
Note in this situation a ring is graded-commutative if and only if the ring is commutative.
These notational changes will allow more ergonomic statements below.

Our goal is to extend the description of the cohomology ring of smooth projective toric varieties in terms of differential operators and the volume polynomial to this situation (see Theorem~\ref{PKh}).
For the reader's convenience, we briefly recall the crucial steps of this description.

Let $V$ be a finite-dimensional vector space over a field $\K$ of characteristic $0$ and let $\Diff(V)$ be the ring of differential operators on $V$ with constant coefficients.
A choice of a basis for $V$, induces an isomorphism of $\Diff(V)$ with the ring of polynomials in symbols $\partial_i$ corresponding to partial derivatives along the $i$-th coordinate (induced by the chosen basis).
The ring $\Diff(V)$ acts naturally on the ring $\Sym(V^*)$ of polynomials on $V$.
In order to describe the cohomology ring of the toric bundle $E_X$ in the case that $c$ is surjective, we use the following theorem which is an analogue of Theorem~\ref{thm-nsdquatient} for commutative algebras which are generated in degree $1$.
It is implicitly contained in the work of Pukhlikov and the second author (see \cite{KP}, see also \cite[Theorem~1.1.]{KavehVolume}).

\begin{theorem}\label{thm-comm-alg}
    Let $A$ be a commutative finite dimensional graded algebra over $\K$.
    Write $A = \bigoplus_{i=0}^{n} A^i$ where $A^i$ is the $i$-th graded piece of $A$.
    Suppose the following conditions hold:
    \begin{enumerate}
    \item
        $A$ is generated (as an algebra) by $A^1$, and $A^0 \simeq A^n \simeq \K$.
    \item
        The bilinear map $A^i \times A^{n-i} \to A^n \simeq \K$ given by $(u,v) \mapsto u\cdot v$ is non-degenerate for all $i = 0, \ldots, n$.
    \end{enumerate}
    If $\rho \colon V\to A^1$ is a surjective linear map, then $A$ is isomorphic to $\Diff(V)/ \Ann(f)$ as graded algebra.
    Here, $f(v) \coloneqq \rho(v)^n\in A^n \simeq \K$ is a homogeneous polynomial of degree $n$ on $V$ and $\Ann(f)=\{\diff\in \Diff(V) \colon \diff f=0\}$ is its annihilator.
\end{theorem}
We provide a proof to illustrate how Theorem~\ref{thm-comm-alg} follows from Theorem~\ref{thm-nsdquatient}.
\begin{proof}
    Since the algebra $A$ is generated in degree $1$, the surjective linear map $\rho \colon V\to A^1$ induces a surjection $p\colon \Diff(V) \simeq \Sym(V) \to A$ of graded algebras.
   By Theorem~\ref{thm-nsdquatient}, there is an isomorphism $A \simeq \Diff(V)/ I(L_\ell)$ where $\ell\colon \Diff(V) \to \K$ is uniquely defined by its values on the $n$-th graded piece $\Diff^n(V)$ of $\Diff(V)$, namely $\ell = \phi \circ p\colon \Diff^n(V)\to \K$ ($\phi\colon A^n \to \K$ a fixed isomorphism).
    It remains to show $I(L_\ell) = \Ann(f)$.
    Since both ideals are homogeneous, it suffices to check $I(I_\ell) \cap \Diff^j(V) = \Ann(f) \cap \Diff^j(V)$ for all $j$. 
    
    Let $v_1,\ldots, v_n\in V$ and recall $L_{v_i}$ denotes the Lie derivative with respect to the vector $v_i$ (see Definition~\ref{def:Lie}).
    The value $\ell(L_{v_1}\cdots L_{v_n})$ at the corresponding differential monomial is given by polarization of the polynomial $f \colon V\to \K, v \mapsto \phi(\rho(v)^n)$ (notice the implicit choice of an isomorphism $\phi\colon A^n \to \K$ in the statement):
    \[
        \ell(L_{v_1} \cdots L_{v_n}) = \frac{1}{n!} L_{v_1}\cdots L_{v_n} f \text{.}
    \]
    Extending by linearity from the product of linear forms to all of $\Diff^n(V)$, we obtain
    \[
        \ell(\diff) = \frac{1}{n!} \diff f \qquad \text{ for any } \diff \in \Diff^n(V) \text{.}
    \]
    Clearly, for any $j>n$, we have $I(L_\ell) \cap \Diff^j(V)= \Ann(f)\cap \Diff^j(V)= \Diff^j(V)$.
    
    Let $d \in \Ann(f)$ be a differential operator of degree $j\le n$.
    Then $d'd f = 0$ for any $d'\in \Diff(V)$, and thus $d\in I(L_\ell)$.
    If $d \notin \Ann(f)$, then $df$ is a non zero polynomial of degree $n-j$.
    There exists $d'\in \Diff^{n-j}(V)$ with $d'df\ne 0$.
    It follows that $\ell(d'd) \ne 0$ which implies $d\notin I(L_\ell)$.
    So $I(L_\ell) = \Ann(f)$ and the theorem is proved.
\end{proof}

Theorems~\ref{thm:BKKdeg2} and~\ref{thm-comm-alg} yield the following description of the cohomology ring $H^*(E_X,\R)$ of toric bundles over a base manifold with cohomology ring generated in degree $2$.

\begin{corollary}\label{cohdeg2}
    Let $p \colon E\to B$ be a torus bundle, $\Sigma$ be a smooth projective fan, and $E_X$ be the associated toric bundle.
    Suppose $H^2(B,\R)$ is generated in degree $2$.
    Then the cohomology ring $H^*(E_X, \R)$ is given by
    \[
        H^{*/2}(E_X,\R)\simeq \Diff(\Am_{M_\R,\Sigma}) / \Ann(I), 
    \]
    where $I \colon \Am_{M_\R,\Sigma} \to \R$ is the homogeneous polynomial of degree $n+s$ given by $I(\Delta) = \int_\Delta \bar c(x)^s \diff \mu$.
\end{corollary}
\begin{proof}
    By Remark~\ref{rem:vwell}, we have  $\Am_{M_\R,\Sigma} = \Pm_{M_\R,\Sigma}\oplus H^2(B,\R)$, and thus the map $\bar \rho \colon \Am_{M_\R,\Sigma} \to H^2(E_X,\R)$ is surjective.
    Moreover, by Theorem~\ref{thm:BKKdeg2}, we have:
    \[
        s! \cdot \bar\rho(\Delta)^{n+s} = (n+s)! \cdot \int_\Delta \bar c(x)^s \diff \mu,
    \]
    so the statement follows from Theorem~\ref{thm-comm-alg}.
\end{proof}

\begin{remark}\label{rem:csurj}
    From Proposition~\ref{cherneq}, it straightforwardly follows that in Corollary~\ref{cohdeg2} we may replace $V = M_\R\oplus H^2(B,\R)$ by $V' = M_\R\oplus c(M_\R)^\perp$ where $c(M_\R)^\perp \subseteq H^2(B,\R)$ is a complementary subspace of $c(M_{\R})$ in $H^2(B,\R)$.
    Then $c(M_\R)^\perp \cong \big(H^2(B,\R)/c(M_\R)\big)$.
    In particular, if $c\colon M_\R\to H^2(B,\R)$ is surjective, the cohomology ring $H^*(E_X, \R)$ is given by
    \[
        H^{*/2}(E_X,\R)\simeq \Diff(\Pm_{M_\R,\Sigma}) / \Ann(I), 
    \]
    where $I \colon \Pm_{M_\R,\Sigma} \to \R$ is the homogeneous polynomial of degree $n+s$ given by $I(\Delta) = \int_\Delta c(x)^s \diff \mu$.
\end{remark}

%%%%%%%%%%%%%%%%%%%%%%%%%%%%%%%%%%%%%%%%%%%%%%%%%%%%%%%%%%%%%%%%%
\section{Examples}
\label{sec:examples}
Here we consider several concrete applications of our results.
Section~\ref{sec:base-var} will enable us to retrieve a version of the Brion-Kazarnovskii theorem \cite{kaz87, bri89} for horospherical varieties.
For the special case of toric bundles over a full flag variety, we obtain two descriptions of the cohomology ring: one via integrals of the Weyl polynomial, and another one relying on the theory of string polytopes.
The section is concluded by a reproduction of the famous description of the cohomology ring of projective bundles.

%%%%%%%%%%%%%%%%%%%%%%%%%%%%%%%%%%%%%%%%%%%%%%%%%%%%%%%%%%%%%%%%%
\subsection{The Brion-Kazarnovskii theorem for horospherical varieties}

Let $G$ be a connected semisimple simply connected group, $T$ be a maximal torus and $B$ be a Borel subgroup containing $T$.
Let $\Phi\subseteq M(T)$ be the corresponding set of roots and $S\subseteq \Phi$ be the set of simple roots corresponding to our choice of Borel $B$.
Recall parabolic subgroups containing $B$ are in one-to-one correspondence with subsets of simple roots.
Finally, let $C^+$ be the positive Weyl chamber in $M_\R(T)$.

Fix a parabolic subgroup $P_I\subseteq G$ containing $B$ given by the set of simple roots $I\subseteq S$.
The inclusion $i=i_I\colon B\to P_I$ gives rise to an inclusion of character lattices
\[
    i^*\colon M(P_I)\to M(B)\simeq M(T),
\]
where the image is the sublattice of $M(T)$ given by (here $\check{\alpha}$ denotes the coroot associated to the simple root $\alpha$)
\[
    i^*(M(P_I)) = \bigcap_{\alpha\in I}\ker\check{\alpha}\subseteq M(T).
\]
From now on we will identify the character lattice $M(P_I)$ with its image $i^*(M(P_I))$ in $M(T)$.
It is well-known that the horospherical homogeneous spaces that are torus principal bundles over $G/P_I$ are classified by sublattices $\Lambda \subseteq M(P_I)$ in the character lattice of $P_I$.
As with $M(P_I)$, we will identify $\Lambda$ with the corresponding sublattice of $M(T)$.
For a given sublattice $\Lambda\subseteq M(P_I)$, the horospherical subgroup $H_\Lambda\subseteq G$ is defined by
\[
    H_\Lambda = \{g\in P \,|\, \lambda(g)=1 \text{ for all } \lambda \in \Lambda\}.
\]
It is easy to see that the normalizer satisfies $N_G(H_\Lambda)= P_I$ and the factor group $T_\Lambda= P_I/H_\Lambda$ is an algebraic torus.
Hence the homogeneous space $G/H_\Lambda$ has a structure of a $T_\Lambda$-principal  bundle over $G/P_I$. 
We identify the character lattice of $T_\Lambda$ with $\Lambda$ via the pullback
\[
    \pi^*\colon M(T_\Lambda) \to M(P_I), \text{ where } \pi\colon P_I\to T_\Lambda \text{ is the natural projection}.
\]

Let $X$ be the $T_\Lambda$-toric variety associated to a fan $\Sigma$ in $\Lambda$ and $E_X$ be the corresponding toric bundle.
Equivalently, $E_X$ is the toroidal compactification of $G/H_\Lambda$ corresponding to $\Sigma$ (regarded as a colored fan).
The section concludes with a computation of the self-intersection polynomial on $H^2(E_X,\R)$ if $\Sigma$ is a smooth projective fan.

Let $c_\Lambda\colon M_\R(T_\Lambda) \to H^2(G/P_I,\R)$ be as before, i.e., the linear map of vector spaces induced by the homomorphism of latices:
\[
    c_\Lambda \colon M(T_\Lambda) \to H^2(G/P_I,\Z), \quad \lambda \mapsto c_1(\cL_\lambda).
\]
Note there is a map $c\colon M_\R(P_I) \to H^2(G/P_I,\R)$ which associates to a character $\lambda \in M_\R(P_I)$ the first Chern class $c_1(\cL_\lambda)\in H^2(G/P_I,\R)$ of the associated line bundle $\cL_\lambda = G \times_{P_I} \C_\lambda$.
For any $\Lambda\subseteq M(P_I)$, the linear map $c_\Lambda\colon M_\R(T_\Lambda)\simeq\Lambda_\R\to H^2(G/P_I,\R)$ coincides with the restriction of the linear map $c\colon M_\R(P_I) \to H^2(G/P_I,\R)$ as $G/H_\Lambda\times_{T_\Lambda}\C_\lambda = G\times_{P_I}\C_\lambda$ for any $\lambda \in \Lambda$.
We will denote all these maps by the same letter $c$.

We recall the following classical result by Borel~\cite{Borel}.
We found the expository paper~\cite{BGG} useful in our study of this beautiful subject.
Let $W = N_G(T)/T$ be the Weyl group of $G$ with respect to $T$.
\begin{theorem}\label{thm:Borel}
    The map $c\colon M_\R(P_I)\to H^2(G/P_I,\R)$ is an isomorphism.
\end{theorem}
\begin{proof}
    By \cite[Proposition~1.3]{BGG} the map $c \colon M_\R(B) \to H^2(G/B,\R)$ is an isomorphism.
    Ineed, by \cite[Proposition~1.3]{BGG}, $c\colon \Sym(M_\R(B))/J\to H^{*/2}(G/B,\R)$ is an isomorphism of graded algebras where J is the ideal generated by $W$-invariant elements in $\Sym(M_\R(B))$ of positive degree.
    However, there are no non-zero $W$-invariant elements in $M_\R(B) \subseteq \Sym(M_\R(B))$, and thus $c \colon M_\R(B)\to H^2(G/B,\R)$ is a linear isomorphism.

    Let $W_I$ be the subgroup of the Weyl group $W$ generated by the reflections $s_\alpha$ for $\alpha \in I$.
    By \cite[Theorem~5.5]{BGG}, the restriction of $c \colon M_\R(B) \to H^2(G/B,\R)$ to $M_\R(B)^{W_I}$ induces an isomorphism
    \[
        c \colon M_\R(B)^{W_I} \xrightarrow{\text{ $\simeq$ }} H^2(G/P_I,\R)\text{.}
    \]
    From above recall
    \[
        M(P_I) = \bigcap_{\alpha \in I} \ker{\check{\alpha}} = \{ \lambda \in M(B) \mid \langle \lambda, \check{\alpha} \rangle = 0 \, \text{for all} \, \alpha \in I\} \text{.}
    \]
    In other words, the sublattice $M(P_I)$ of $M(B)$ is obtained by intersecting $M(B)$ with the intersection of the reflection hyperplanes of the reflections $s_\alpha$ for $\alpha \in I$.
    This shows $M_\R(B)^{W_I} = M_\R(P_I)$ as desired.
\end{proof}

For a dominant weight $\lambda\in C^+$, let  us denote by $V_\lambda$ the irreducible representation of $G$ with highest weight $\lambda$.
We will also need the following well-known result.
\begin{theorem}[Borel-Weil-Bott theorem, see, e.g., {\cite[Theorem~6.1]{BorelWeil}}]
    \label{thm:BorelWeil}
    We have
    \[
        H^0(G/P_I,\cL_\lambda) = V_\lambda \qquad \text{for any dominant weight $\lambda \in C^+\cap M(P_I)$.}
    \]
\end{theorem}

Weyl showed (see, for instance, \cite[Satz 5]{Weyl}) that the map $\lambda\mapsto \dim(V_\lambda)$ is the restriction of an explicit (inhomogeneous) polynomial on $M_\R(T)$ to the set of dominant weights.
Let $f_{W,I}$ be the top homogeneous part of the restriction of the Weyl polynomial to $M_\R(P_I)\subseteq M_\R(T)$.
We will need the following classical result.
\begin{proposition}\label{prop:degreeP}
    Let $G, P_I$ and $f_{W,I}$ be as before.
    Then for any $\lambda\in M_\R(P_I)$ we have
    \[
        c(\lambda)^s = s!\cdot f_{W,I}(\lambda) \qquad \text{where $s=\dim_\C(G/P_I)$.}
    \]
\end{proposition}
\begin{proof} 
    Since both sides of the equality are polynomials, it suffices to verify it on a subset of dominant weights whose convex hull is an affine full-dimensional cone in $M_\R(P_I)$.
    It's well-known that the line bundle $\cL_\lambda$ associated to $\lambda \in M(P_I)$ is very ample if and only if $\langle \lambda, \check{\alpha} \rangle > 0$ for all $\alpha \in S\setminus I$, i.e., the interior lattice points in the cone $C^+\cap M_\R(P_I)$ yield very ample line bundles.
    It suffices to verify the above equality on this discrete set of points.
 
 	Fix $\lambda \in C^+\cap M_\R(P_I)$ in the interior.
	The very ample line bundle $\cL_\lambda$ gives rise to an embedding of $G/P_I$ into the projective space $\mathbb{P}(V_\lambda^*)$ (see Theorem~\ref{thm:BorelWeil}).
    Note the degree of $\cL_\lambda$, i.e., $c(\lambda)^s$, coincides with the degree of the projective variety $G/P_I \hookrightarrow \mathbb{P}(V_\lambda^*)$.
    Let $H_\lambda(t)$ be the Hilbert polynomial of the section ring $R_\lambda = \bigoplus_{k\ge0} H^0(G/P_I,\cL_\lambda^{\otimes k})$.
    Then the degree of the projective variety $G/P_I \hookrightarrow \mathbb{P}(V_\lambda^*)$ coincides with the $s!$ multiple of the leading coefficient of $H_\lambda(t)$.
    By Theorem~\ref{thm:BorelWeil}, $R_\lambda = \bigoplus_{k\ge0} V_{k\lambda}$, and thus $H_\lambda(t)$ coincides with the Weyl polynomial, say $f_W$, restricted to the ray $\R_{\ge0}\cdot\lambda$, i.e., $f_W(t\lambda) = H_\lambda(t)$.
    The restriction $f_W|_{M_\R(P_I)}$ accepts the decomposition into homogeneous components
    \[
        f_W|_{M_\R(P_I)} = f_{W,I} + \text{homogeneous components of lower degree}.
    \]
    In particular,
    \[
        f_W(t\lambda) = f_W|_{M_\R(P_I)}(t\lambda) = f_{W,I}(\lambda)t^{\dim(G/P_I)} + \text{lower terms}.
    \]
    Note $f_{W,I}(\lambda)\neq 0$ as $f_{W,I}$ is of top degree (i.e., of degree $\dim(G/P_I)$).
    The statement follows.
\end{proof}

We are now ready to formulate a version of the Brion-Kasarnovskii theorem \cite{kaz87, bri89} for our setting.
Let us consider a pair of vector spaces $\Lambda_\R\subseteq M_\R(P_I)$.
By Theorem~\ref{thm:Borel}, we can identify $M_\R(P_I)$ with $H^2(G/P_I,\R)$.
Let $\Lambda_\R^\perp$ be a complementary subspace of $\Lambda_\R$ in $M_\R(P_I)$.
By Remark~\ref{rem:csurj}, we have the surjection
\[
    \bar\rho\colon\Am_{\Lambda_\R,\Sigma}\simeq \Pm_{\Lambda_\R,\Sigma} \oplus \Lambda_\R^\perp \to H^2(E_X,\R). 
\]
Hence, $H^*(E_X,\R) = \Diff(\Am_{\Lambda_\R,\Sigma})/\Ann(f)$ where $f \colon \Am_{\Lambda,\Sigma} \to H^{2(n+s)}(E_X,\R)\simeq\R$ is the self-intersection polynomial given by $f(\Delta) = \bar\rho(\Delta)^{n+s}$ (see Theorem~\ref{thm-comm-alg}).
It remains to compute $f$.
\begin{theorem}\label{thm:brikaz}
    In the notation as above, for any $\Delta\in\Am_{\Lambda_\R,\Sigma}$ we have
    \[
        \bar\rho(\Delta)^{n+s} = (n+s)!\cdot\int_\Delta  f_{W,I}(\lambda) \diff \mu,
    \]
    where $n=\rk(\Lambda)$ and $\mu$ is Lebesgue measure on $\Lambda_\R$ normalized with respect to $\Lambda$.
\end{theorem}
\begin{proof}
    The theorem follows immediately from Theorem~\ref{thm:BKKdeg2} and Proposition~\ref{prop:degreeP}. 
\end{proof}

\subsection{Horospherical varieties over full flag variety}
\label{sec:horofull}
As before, let $G$ be a connected semisimple simply connected algebraic group, $B$ a Borel subgroup in $G$, and $T\subseteq B$ a maximal torus.
In the proof of Theorem~\ref{thm:Borel}, we already used Borel's famous description of the cohomology ring of the full flag variety $G/B$~\cite{Borel}.
For the reader's convenience we briefly recall it: Borel showed that the cohomology ring of the full flag variety $G/B$ is isomorphic to the ring of coinvariants of the Weyl group $W=N_G(T)/T$ of $G$:
\[
  H^*(G/B, \R) \simeq \Sym(M_\R) / \Sym(M_\R)_{> 0}^W \text{,}
\]
where the action of $W$ on the ring of polynomials on the Lie algebra of $T$ is given by the coadjoint action on~$M_\R$.
In particular, $H^*(G/B, \R)$ is generated in degree 2. 

Let $U\subseteq B$ be the unipotent radical of $B$ and let $U\subseteq H\subseteq B$ be any closed subgroup of $G$. Then, one has $N_G(H) =B$ and hence there is a natural $H/B$-principal bundle:
\[
  p\colon E = G/H \to G/B \text{.}
\]
The group $B/H$ is an algebraic torus (it is a factor group of $B/U\simeq (\C^*)^{\rk(G)}$), so $p\colon E \to G/B$ is a torus principal bundle. In what follows we identify $B/U$ with the maximal torus $T$ via natural projection $T\to B/U$.

As in Section~\ref{sec:base-var}, we can apply Theorem~\ref{thm:brikaz} to compute the cohomology ring $H^*(E_X,\R)$ of the toric bundle associated to a smooth projective fan $\Sigma$ in $M_\R(B/H)$.
Denote by $f_W$ the top homogeneous part of the Weyl polynomial on $M_\R(T)$.
Then Proposition~\ref{prop:degreeP} becomes:
\begin{proposition}\label{prop:degree}
    Let $G, T$ and $f_W$ be as before.
    Then for any $\lambda\in M_\R(T)$ we have
    \[
        c(\lambda)^s = s!\cdot f_W(\lambda) \qquad \text{where $s=\dim_\C(G/B)$.}
    \]
\end{proposition}

Let $L$ be the image of $M_\R(B/H)$ in $M_\R(T)$ under the pull-back of characters and let $\mu$ be a translation invariant Lebesgue measure on $L$ normalized with respect to the lattice $M(B/H)$. 
\begin{theorem}\label{thm:G/HWeyl}
    Consider the torus bundle $E = G/H \to G/B$, for $U\subseteq H \subseteq B$.
    Let $\Sigma \subseteq M_\R(B/H)$ be a smooth projective fan with associated toric variety $X$.
    Then the cohomology ring $H^*(E_X, \R)$ is given by
    \[
        H^{*/2}(E_X,\R)\simeq \Diff(\Am_{L,\Sigma}) / \Ann(I) \text{,}
    \]
    where $I(\Delta) = \int_\Delta f_W(\lambda) \diff \mu$.
    % where $I(\Delta) = (n+s)!\int_\Delta f_W \diff \mu$.
    Furthermore, the ring of conditions $C_\R^*(G/U)$ is given by
    \[
        C^{*}_\R(G/U)\simeq \Diff(\Am_L) / \Ann(I) \text{.}
    \]
\end{theorem}
\begin{proof}
    The theorem follows immediately from Corollary~\ref{cohdeg2} and Proposition~\ref{prop:degree}.
\end{proof}

%%%%%%%%%%%%%%%%%%%%%%%%%%%%%%%%%%%%%%%%%%%%%%%%%%%%%%%%%%%%%%%%%%%%%

\subsection{Relation to string polytopes}
\label{sec:string-poly}
In the next two subsections we continue to examine the example of toroidal horospherical varieties over full flag varieties.
In this case, Theorems~\ref{thm:brikaz} and~\ref{thm:G/HWeyl} have beautiful reformulations.
For simplicity of the exposition we will start with the case $G=\SL_n$ and $H=U\subseteq \SL_n$.

For $\lambda$ a dominant weight, let $\GZ_\lambda \subseteq \R^{N=n(n-1)/2}$ be the corresponding \emph{Gelfand-Zetlin polytope} \cite{GZ}.
From the defining inequalities, it can be straightforwardly deduced that Gelfand-Zetlin polytopes are additive:
\[
  \GZ_{\lambda+\mu}=\GZ_\lambda+\GZ_\mu \qquad \text{for any dominant weights $\lambda, \mu$}\text{.}
\]
By linearity, we extend the definition of Gelfand-Zetlin polytopes to all of the positive Weyl chamber $C^+$, i.e. we define $\GZ_\lambda$ for any $\lambda \in C^+$:
if $\lambda$ is in the positive Weyl chamber, we can write $\lambda = \sum_{i=1}^r a_i \cdot \mu_i$ for positive real numbers $a_i$ and dominant (integral) weights $\mu_i$; set $\GZ_\lambda = \sum_{i=1}^r a_i \cdot \GZ_{\mu_i}$ (Minkowski addition).
By the additivity of Gelfand-Zetlin polytopes, it follows that this definition is well-defined.

The Gelfand-Zetlin polytopes form a linear family of polytopes, a concept which we recall from \cite[Definition 1.2]{kavvil}.
Let $V, W$ be two real vector spaces (not necessarily finite-dimensional) and let $C \subseteq V$ be a full-dimensional convex cone.
Note that the cone $C$ isn't necessarily closed.
We only assume that $c_1 x_1 + c_2 x_2 \in C$ for any $x_1, x_2 \in C$ and all non-negative real numbers $c_1, c_2 \ge 0$.
An $\R$-linear map $\Delta \colon C \to \Pm_W^+$ is called a \emph{linear family of polytopes} where ``\emph{linearity}'' is explained in the following folklore result:
\begin{proposition}
  \label{prop:linearity}
  A map $\Delta \colon C \to \Pm_W^+$ is called linear if either of the following equivalent conditions is satisfied:
  \begin{enumerate}
  \item $\Delta(c_1 x_1 + c_2 x_2) = c_1 \Delta(x_1) + c_2 \Delta(x_2)$ for any $x_1,x_2 \in C$ and all non-negative real numbers $c_1,c_2 \ge 0$.
  \item There is a (unique) linear map $V \to \Pm_W$ such that its restriction to $C$ coincides with $\Delta(\cdot)$.
  \end{enumerate}
\end{proposition}
For the reader's convenience, we include a proof for Proposition~\ref{prop:linearity}.
\begin{proof}
  The implication ``(2) $\Rightarrow$ (1)'' is straightforward.
  It remains to show the reverse implication.
  
  Choose a (Hamel) basis of $V$, say $\{ v_i \}_{i \in I}$.
  Since $C$ is full-dimensional, we may (and will) assume that $\{ v_i \}_{i \in I} \subseteq C$.
  We define a linear map $F \colon V \to \Pm_W$ by setting $F(v_i) = \Delta(v_i)$ for $i \in I$ and show that $F|_C = \Delta$.
  Let $x \in C$ expressed as $x = \sum_{j \in J} \lambda_j x_j$ for a finite subset $J \subseteq I$ and $\lambda_j \in \R$.
  Consider the partition of $J$:
  \[
    J_- \coloneqq \{ j \in J \colon \lambda_j < 0\} \qquad \text{and} \qquad J_+ \coloneqq \{ j \in J \colon \lambda_j > 0\} \text{.}
  \]
  Then $x + \sum_{j \in J_-} (-\lambda_j) v_j = \sum_{j \in J_+} \lambda_j v_j$, and thus
  \[
    \Delta\rleft(x + \sum_{j \in J_-} (-\lambda_j) v_j\rright) = \Delta\rleft(\sum_{j \in J_+} \lambda_j v_j\rright) \Rightarrow \Delta(x) + \sum_{j \in J_-} (-\lambda_j) \Delta(v_j) = \sum_{j \in J_+} \lambda_j \Delta(\lambda_j) \Rightarrow \Delta(x) = \sum_{j \in J} \lambda_j \Delta(v_j) \text{,}
  \]
  where the second implication follows by our assumption that (1) holds.
  The statement follows by the observation that $F(x) = \sum_{j \in J} \lambda_j \Delta(v_j)$.
\end{proof}

\begin{example}
  It is straightforward to show that Gelfand-Zetlin polytopes $GZ \colon (C^+)^\circ \to \Pm_{\R^N}^+$ form a linear family of polytopes.
  Here, $(C^+)^\circ$ denotes the interior of the positive Weyl chamber.
\end{example}

For a linear family of polytopes $\Delta\colon C \to \Pm_W^+$, we construct its \emph{lift} $\widetilde \Delta\colon \mathscr{C} \to \Pm_{V\oplus W}^+$ where $\mathscr{C}$ denotes the following convex cone ($C^\circ$ denotes the interior of $C$)
\[
  \mathscr{C} \coloneqq \{ Q \in \Pm_V^+ \colon Q \subseteq C^\circ \} \text{.}
\]

\begin{proposition}
  $\mathscr{C}$ is a full-dimensional convex cone.
\end{proposition}
\begin{proof}
  A straightforward calculation shows that $\mathscr{C}$ is a convex cone.
  To show that $\mathscr{C} \subseteq \Pm_V$ is full-dimensional, let $Q = Q_1 - Q_2 \in \Pm_V$ where $Q_1, Q_2 \in \Pm_V^+$.
  Since $C \subseteq V$ is full-dimensional, there is a translation vector $v \in V$ such that $Q_1 + v, Q_2 + v \subseteq C$.
  The statement follows from the observation that $Q = (Q_1 + v) - (Q_2 + v)$.
\end{proof}

We define the lift $\widetilde \Delta \colon \mathscr{C} \to \Pm_{V \oplus W}^+$ of the linear family of polytopes $\Delta \colon C \to \Pm_W^+$ by setting for any $Q \in \mathscr{C}$:
\[
  \widetilde\Delta(Q) \coloneqq \{ (x,y) \in V \oplus W \colon x \in Q, y \in \Delta(x) \} \text{.}
\]
\begin{proposition}
  Suppose $V$ is finite-dimensional.
  Then for any $Q \in \mathscr{C}$, $\widetilde \Delta (Q)$ is a polytope, i.e. $\widetilde \Delta(Q) \in \Pm_{V\oplus W}^+$.
  Furthermore, $\widetilde \Delta \colon \mathscr{C} \to \Pm_{V \oplus W}^+$ is linear in the sense of Proposition~\ref{prop:linearity}.
\end{proposition}
\begin{proof}
  Let $Q \in \mathscr{C}$.
  Consider the $H$-representation of $Q$, i.e.
  \[
    Q = \bigcap_{i=1}^s \{ x \in V \colon \langle x, n_i \rangle \le b_i \} \qquad \text{where $n_i \in V^*$ are the facet normals and $b_i \in \R$.}
  \]
  To determine the $H$-representation of $\widetilde \Delta(Q)$, we also need the $H$-representation of $\Delta(x)$ for $x \in C^\circ$.
  By \cite[Proposition 1.3]{kavvil}, the polytopes $\Delta(x)$ for $x \in C^\circ$ share the same normal fan (indeed, Kaveh and Villella assume $C$ to be polyhedral, but their argument can be easily generalized for arbitrary convex cones).
  In particular,
  \[
    \Delta(x) = \bigcap_{j=1}^\sigma \{ y \in W \colon \langle y, \nu_j \rangle \le \beta_j(x)\} \qquad \text{where $\nu_j \in W^*$ are the facet normals and $\beta_j(x) \in \R$.}
  \]
  Since for $x, x' \in C^\circ$, $\Delta(x)$, $\Delta(x')$ and $\Delta(x) + \Delta(x')$ share the same normal fan (see for instance \cite[Proposition 7.12]{Ziegler}), it straightforwardly follows from the linearity of $\Delta$ that the $\beta_j\colon C^\circ \to \R$ are linear in the sense of Proposition~\ref{prop:linearity}.
  In particular, $\beta_j$ can be uniquely extended to a linear function on $V$, i.e. $\beta_j \in V^*$.
  
  From the above it follows that $\widetilde\Delta(Q) \subseteq V \oplus W$ is a polytope as $\widetilde\Delta(Q)$ is bounded with $H$-representation:
  \[
    \widetilde\Delta(Q) = \{ (x,y) \in V\oplus W \colon \langle x, n_i \rangle \le b_i \; \text{for $i=1, \ldots, s$ and} \; \langle y, \nu_j \rangle \le \langle x, \beta_j \rangle \; \text{for $j=1, \ldots, \sigma$}\} \text{.}
  \]
  Hence, $\widetilde\Delta(Q) \in \Pm_{V\oplus W}^+$.
  By using the linearity of $\Delta$ (in the sense of Proposition~\ref{prop:linearity}) a straightforward calculation shows that $\widetilde\Delta \colon \mathscr{C} \to \Pm_{V\oplus W}$ is linear too.
\end{proof}

We denote the composition of the (linear) lift $\widetilde \Delta \colon \mathscr{C} \to \Pm_{V \oplus W}^+$ with the volume polynomial on $\Pm_{V \oplus W}$ by $\widetilde \Vol$.
The polynomial $\widetilde \Vol$ has degree $\dim(V) + \dim(W)$.

Let us apply this construction of linear families of virtual polytopes to the linear family of Gelfand-Zetlin polytopes $GZ \colon C^+ \to \Pm_{\R^N}^+$.
Let $B, T \subseteq \SL_n$ and $M = M(T)$ be as before.
The polynomial $\widetilde \Vol \colon \Pm_{M_\R} \to \R$ has degree $\dim(M_\R) + N = \dim_\C(E)$, and thus its degree coincides with the degree of the polynomial $I(\Delta)$.
Indeed, these two polynomials coincide up to a constant scalar multiple.
We get the following result.

\begin{theorem}\label{thm:brikazGL/U}
    Consider the torus bundle $E = \SL_n/U \to \SL_n/B$.
    Let $\Sigma\subseteq M_\R$ be a smooth projective fan and let $\Delta\in \Pm_{M_\R,\Sigma}$.
    Then we have 
    \[
        \rho(\Delta)^{n+N-1} = (n+N-1)!\cdot \widetilde\Vol(\Delta),
    \]
    where $N=\dim(\SL_n/B)$.
\end{theorem}
\begin{proof}
    Since Gelfand-Zetlin polytopes are linear in $\lambda$ and $|\GZ_\lambda\cap M|=\dim(V_\lambda)$, we have $\Vol(\GZ_\lambda) = f_W(\lambda)$. Hence by Proposition~\ref{prop:degree}
    \[
        c(\lambda)^N = N! \cdot \Vol(\GZ_\lambda).
    \]
    and hence $\int_\Delta c(\lambda)^N \diff \mu = N!\cdot\widetilde \Vol(\Delta)$.
    The statement follows by Theorem~\ref{thm:brikaz}.
\end{proof}

As an immediate corollary of Theorem~\ref{thm:brikazGL/U} we obtain the following result. 

\begin{theorem}
  \label{thm:GL/U}
  Consider the torus bundle $E = \SL_n/U \to \SL_n/B$.
  Let $\Sigma\subseteq M_\R$ be a smooth projective fan with corresponding toric variety $X$.
  Then the cohomology ring $H^*(E_X,\R)$ is given by
  \[
    H^{*/2}(E_X, \R)\simeq \Diff(\Pm_{M_\R,\Sigma}) / \Ann(\widetilde \Vol) \text{.}
  \]
    The ring of conditions $C^*_\R(\SL_n/U)$ is given by
  \[
    C^{*}_\R(\SL_n/U)\simeq \Diff(\Pm_{M_\R}) / \Ann(\widetilde \Vol) \text{.}
  \]
\end{theorem}
\begin{proof}
    The statements follow from Theorem~\ref{thm:brikazGL/U} and Remark~\ref{rem:csurj} as the map $c\colon M_\R(T) \to H^2(\SL_n/B,\R)$ is surjective by Borel's theorem.
\end{proof}

%%%%%%%%%%%%%%%%%%%%%%%%%%%%%%%%%%%%%%%%%%%%%%%%%%%%%%%%%%%%%%%%%

\subsection{\texorpdfstring{General toric bundles over a full flag variety $G/B$}{General toric bundles over a full flag variety G/B}}
\label{sec:tb-over-flag}
In general, suppose $G$ is an (arbitrary) connected semisimple simply connected group and let $U\subseteq B\subseteq G$ be as above.
In this section we will give a general construction for toric bundles $G/H\to G/B$, with $U\subseteq H\subseteq B$.
We use the same letter $N$ to denote the (complex) dimension of the full flag variety $G/B$.
As above $M=M(T)$ denotes the character lattice of $T$.

For any character $\lambda$ in the positive Weyl chamber, one defines a \emph{string polytope} $S_\lambda \subseteq \R^N$ such that
\[
    |S_\lambda\cap M(T)| = \dim(V_\lambda) \text{.}
\]
The construction of string polytopes depends on the choice of a reduced expression $\underline w_0$ of the longest element $w_0$ in the Weyl group, so fix some reduced expression $\underline w_0$.
There exists a convex cone $\Cm_{\underline w_0}\subseteq M_\R\oplus \R^N$ called \emph{weighted string cone} such that 
\[
  S_\lambda = \pi^{-1}(\lambda)\cap \Cm_{\underline w_0} \text{,}
\]
where $\pi \colon M_\R\oplus \R^N\to M_\R$ is the natural projection (see \cite{littelmanncones, BerZel}).
Like Gelfand-Zetlin polytopes, string polytopes are homogeneous in $\lambda$, i.e. $S_{k\lambda}=kS_\lambda$.
However, string polytopes $S_\lambda$ are not necessarily additive, i.e. $S_{\lambda+\mu}=S_\lambda + S_\mu$ might not be true.
As in the case of $\SL_n$ we get $c(\lambda)^N = N!\cdot \Vol(S_\lambda)$ (see also \cite[Remark 3.9 (iii)]{alexeevbrion} or \cite[Theorem 2.5]{KavehVolume}).

By \cite[Proposition 1.4]{kavvil}, it follows that the map which sends weights to the corresponding string polytopes is piecewise linear with respect to a fan $\Fm$ which subdivides the positive Weyl chamber.
That is the restriction of the map $S \colon C^+ \to \Pm_{\R^N}; \lambda \mapsto S_\lambda$ to every cone of $\Fm$ is linear.
In particular, if $\sigma$ is a full dimensional cone of $\Fm$, then $S \colon \sigma \to \Pm_{\R^N}^+$ is a linear family of polytopes.
By the construction from above, we obtain its corresponding lift $\Pm_{M_\R} \to \Pm_{M_\R \oplus \R^N}$ which gives rise to a polynomial $\widetilde \Vol$ of degree $\dim(M_\R) + N = \dim(G/U)$.

Now, let $W$ be the image of $M_\R(B/H)$ in $M_\R(T)$ and let $\mu$ be the translation invariant measure on $W$ normalized with respect to the lattice $M(B/H)$.
Since the space of affine virtual polytopes $\Am_{W}$ is a natural subspace of $\Pm_{M_\R(T)}$, their lifts are well-defined. 
In particular, for a virtual polytope $\Delta\in \Am_{W}$, its lift $\widetilde\Delta$ is a virtual polytope in a $(\dim W + N)$-dimensional space.
Let us denote by $\widetilde \Vol$ its $(\dim W + N)$-dimensional volume induced by $\mu$.
The same arguments as above yield generalizations of Theorems~\ref{thm:brikazGL/U} and~\ref{thm:GL/U}:

\begin{theorem}\label{thm:G/U}
    Consider the torus bundle $E = G/H \to G/B$.
    Let $\Sigma\subseteq M_\R(B/H)$ be a smooth projective fan with associated toric variety $X$.
    Let $W\subseteq M_\R(T)$ be as before and let $\Delta\in \Am_{W,\Sigma}$ be an affine virtual polytope.
    Then we have
    \[
    	\bar\rho(\Delta)^{r+N}=(r+N)!\cdot \widetilde\Vol(\Delta),
    \]
    Where $r=\dim(M_\R(B/H))=\dim(W)$. Moreover, the cohomology ring $H^*(E_X, \R)$ is given by
    \[
        H^{*/2}(E_X,\R)\simeq \Diff(\Am_{W,\Sigma}) / \Ann(\widetilde \Vol) \text{.}
    \]
    and the ring of conditions $C^*(G/H)$ is given by
    \[
        C^{*}_\R(G/H)\simeq \Diff(\Am_{W}) / \Ann(\widetilde \Vol) \text{.}
    \]
\end{theorem}

In this case the ring of conditions $C^*_\R(G/H)$ coincides with the ring of complete intersections defined in \cite{KaveKhovanskii} and our description of $C^*_\R(G/H)$ is analogous to the one given there.

\begin{remark}
  Note that the lift of the linear family of string polytopes depends on the choice of a maximal cone $\sigma$ of $\Fm$.
  However, since for $\lambda$ in the positive Weyl chamber the volume of $S_\lambda$ coincides (up to a constant multiple) with the degree of the corresponding line bundle $\cL_\lambda$, the lifted volume polynomial $\widetilde \Vol$ is independent of the choice of $\sigma$ of $\Fm$.
  Indeed, for polytopes $Q$ in the interior of the positive Weyl chamber, $\widetilde \Vol(Q) = \int_Q \frac{1}{N!} c(\lambda)^N \diff \lambda$.
\end{remark}

%%%%%%%%%%%%%%%%%%%%%%%%%%%%%%%%%%%%%%%%%%%%%%%%%%%%%%%%%%%%%%%%%

\subsection{Projective bundles}
Here we prove the classical description of the cohomology ring of projective bundles using our computations of cohomology rings of toric bundles (see for instance \cite[Theorem 3.3]{fulton}). 
\begin{theorem}
    Let $B$ be a smooth orientable closed manifold and let $V\to B$ be a complex vector bundle of rank $n$.
    Consider the projective bundle compactification $E_V = \p(E\oplus \Om)$ of $V$ ($\Om$ denotes the trivial line bundle on $B$).
    Then the cohomology ring of $\p(E\oplus \Om)$ can be computed as:
    \[
        H^*(E_V, \R)\simeq \frac{H^*(B, \R)[t]}{t^{n+1}+c_1(V\oplus \Om)\cdot t^n +\ldots + c_n(V\oplus \Om)\cdot t+ c_{n+1}(V\oplus \Om)} \text{.}
    \]
\end{theorem} 
Here $c_1(V\oplus \Om), \ldots, c_{n+1}(V\oplus \Om)$ denote the Chern classes of the vector bundle $V \oplus \Om$ of rank $n+1$ over $B$.
\begin{proof}
    By the Leray-Hirsch theorem $H^*(E_V, \R)\simeq H^*(B,\R)\otimes H^*(\p^n,\R)$ and in particular there is a surjection 
    \[
        \phi\colon H^*(B, \R)[t] \to H^*(E_V, \R).
    \]
    We show that $\ker(\phi)  = \langle P(t) = t^{n+1}+c_1(V\oplus \Om)\cdot t^n +\ldots + c_n(V\oplus \Om)\cdot t+ c_{n+1}(V\oplus \Om) \rangle$.
    Since by Leray-Hirsch
    \[
        \dim_\R \left(\frac{H^*(B, \R)[t]}{t^{n+1}+c_1(V\oplus \Om)\cdot t^n +\ldots + c_n(V\oplus \Om)\cdot t+ c_{n+1}(V\oplus \Om)}\right) = \dim_\R (H^*(E_V, \R)),
    \]
    it remains to show that $P(t)\in \ker(\phi)$.
    By the splitting principle it is enough to show that $P(t)\in \ker(\phi)$ for vector bundles $V$ that are direct sums of line bundles $V\simeq \cL_1\oplus\ldots\oplus\cL_n$ on $B$.
    
    In this case, let $E\to B$ be a toric bundle which corresponds to $V\simeq \cL_1\oplus\ldots\oplus\cL_n$.
    Then $E_V$ is a toric bundle given by a fan $\Sigma$ of the projective space $\p^n$.
    After an appropriate choice of coordinates in $N \simeq \Z^n$ (the lattice of $1$-parameter subgroups), the fan $\Sigma$ coincides with the normal fan of the standard simplex.
    More precisely, $\Sigma$ has rays generated by $e_1,\ldots, e_n, e_{n+1}$, where 
    \[
        e_1 = (1,0,\ldots,0), \ldots, e_n = (0,\ldots,0,1), e_{n+1} = (-1,\ldots,-1) \text{,}
    \]
    and the cones of $\Sigma$ are comprised by the conical hull of any subset of these rays. 
    
    By Theorem~\ref{thm:US}, the cohomology ring of $E_V$ can be computed as
    \[
        H^*(E_V, \R) \simeq \frac{H^*(B, \R)[H_1,\ldots,H_{n+1}]}{\langle H_1\cdots H_{n+1}\rangle + \langle c(\lambda_i) - H_i + H_{n+1} \colon i=1, \ldots, n\rangle}\simeq \frac{H^*(B, \R)[H_{n+1}]}{\left(H_{n+1}+c(\lambda_1)\right) \cdots \left(H_{n+1}+c(\lambda_n)\right)} \text{,}
    \]
    where $\lambda_1, \ldots, \lambda_n$ denotes the basis of $M$ (the character lattice of the torus) dual to $e_1, \ldots, e_n$.
    
    It is left to notice that since $V = \cL_{\lambda_1} \oplus \ldots \oplus \cL_{\lambda_n}$, the Chern classes of $V\oplus \Om$ can be computed as 
    \[
        c_i(V\oplus \Om) = s_i(c(\lambda_1),\ldots,c(\lambda_n)),
    \]
    where $s_i$ is the $i$-th elementary symmetric polynomial in $n$ variables.
    So the theorem follows from the fact that 
    \[
        \left(H_{n+1}+c(\lambda_1)\right) \cdots \left(H_{n+1}+c(\lambda_n)\right)= H_{n+1}^{n+1} + s_1(c(\lambda_1),\ldots,c(\lambda_n))\cdot H_{n+1}^n  + \ldots + s_n(c(\lambda_1),\ldots,c(\lambda_n))\cdot H_{n+1} \text{.}
    \]
\end{proof}

\providecommand{\bysame}{\leavevmode\hbox to3em{\hrulefill}\thinspace}

\end{document}